\documentclass[12pt]{amsart}
\usepackage{amssymb, latexsym, amsthm, enumitem}%
\usepackage{color, graphicx}

\newcommand\operA[2]{{\if!#2!\operatorname{#1}\else{\operatorname{#1}_{#2}^{\phantom{I}}}\fi}} %

\DeclareMathOperator{\Gal}{Gal} %

\newcommand\suchthat{{\,:\ \,}}

\newcommand{\card}[1]{{\left|{#1}\right|}}

\newcommand\isom{{\,\cong\,}}
\newcommand\tensor[1][{}]{{\otimes_{#1}}}
\DeclareMathOperator{\Ker}{Ker} %
\newcommand\ideal[1]{{\left<{#1}\right>}}
\newcommand\sg[1]{{\ideal{#1}}}
\def\sub{\subseteq}

\def\R{{\mathbb{R}}}
\def\a{{\mathfrak{a}}}
\def\co{{\,{:}\,}}
\def\s{{\sigma}}
\newcommand{\set}[1]{{\left\{#1\right\}}}

\def\ra{{\rightarrow}}
\def\hra{{\hookrightarrow}}

\renewcommand{\Im}{\operatorname{Im}} %
\renewcommand{\Ker}{\operatorname{Ker}} %

\newcommand\dimcol[2]{{[{#1}\!:\!{#2}]}} %

\newcommand\Zent{{\operatorname{Z}}}
\def\({\left(}
\def\){\right)}

\newcommand\cat[1]{{\mathbf{#1}}}
\newcommand\abs[1]{{\left|{#1}\right|}}
\newcommand\obj[1][]{{\operatorname{obj}\if!#1!\else(\cat{#1})\fi}}
\newcommand\Obj[1][]{{\operatorname{obj}\if!#1!\else(\cat{#1})\fi}}

\newcommand\lam{{\lambda}}

\newcommand\eq[1]{{(\ref{#1})}}%

\newcommand\Eq[1]{{Equation \eq{#1}}}

\long\def\forget#1\forgotten{}
\long\def\second#1\seconded{#1}

\newcommand{\diag}{\operatorname{diag}}
\def\E{{\mathcal{E}}}
\newcommand\normali{{\lhd}}
\newtheorem{thm}{Theorem}[section] %

\newtheorem{cor}[thm]{Corollary}
\newtheorem{defn}[thm]{Definition}
\newtheorem{exer}[thm]{Exercise}
\newtheorem{exmpl}[thm]{Example}

\newtheorem{lem}[thm]{Lemma}
\newtheorem{prop}[thm]{Proposition}

\newtheorem{rem}[thm]{Remark}

\newcommand\Cref[1]{{Corollary~\ref{#1}}}

\newcommand\Dref[1]{{Definition~\ref{#1}}}

\newcommand\Eref[1]{{Example~\ref{#1}}}
\newcommand\Lref[1]{{Lemma~\ref{#1}}}
\newcommand\Pref[1]{{Proposition~\ref{#1}}}
\newcommand\Prefs[2]{{Propositions~\ref{#1} and~\ref{#2}}}
\newcommand\Rref[1]{{Remark~\ref{#1}}}
\newcommand\Tref[1]{{Theorem~\ref{#1}}}

\newcommand\Sref[1]{{Section~\ref{#1}}}
\newcommand\Srefs[2]{{Sections~\ref{#1} and~\ref{#2}}}
\newcommand\Ssref[1]{{Subsection~\ref{#1}}}

\newcommand\Exeref[1]{{Excercise~\ref{#1}}}

\newcommand\etale{{\'{e}tale}}
\newcommand\End{{\operatorname{End}}}
\def\eg{{{\it{e.g.}}}}

\newcommand\MM[1]{\mathcal{M}^{\operatorname{#1}}}

\renewcommand\SS[1]{\mathcal{S}^{\operatorname{#1}}}
\def\w{{\omega}}
\newcommand\Jac{{\operatorname{J}}}

\newcommand\sa{{\operatorname{sa}}}
\newcommand\nacl[1]{{[#1]^\sa}}

\newcommand\Br[1][]{{\operatorname{Br}^{\operatorname{#1}}}}
\newcommand\allBr[1][F]{{\overline{\Br}(#1)}}

\newcommand\Brp[1][p]{{{}_{{#1}\!\!\!\:}\Br}} %
\newcommand\Ce[2][A]{{\operatorname{C}_{#1}(#2)}} %
\newcommand\nuc[2][]{{\mathfrak{N}_{#1}(#2)}}            %

\newcommand\mul[1]{{#1^{\times}}} %
\def\R{{\mathbb{R}}}
\def\HQ{{\mathbb{H}}}
\def\C{{\mathbb{C}}}
\def\Q{{\mathbb{Q}}}

\renewcommand\H[4][!]{{\operatorname{H}^{#2}\!\!\;({#3},{#4}{\if!#1\relax\else(#1)\fi})}}
\newcommand\HB[4][!]{{\operatorname{B}^{#2}\!\!\;({#3},{#4}{\if!#1\relax\else(#1)\fi})}}
\newcommand\HC[4][!]{{\operatorname{C}^{#2}\!\!\;({#3},{#4}{\if!#1\relax\else(#1)\fi})}}
\newcommand\HZ[4][!]{{\operatorname{Z}^{#2}\!\!\;({#3},{#4}{\if!#1\relax\else(#1)\fi})}}
\newcommand\M[1][d]{{\operatorname{M}_{#1}}} %

\newcommand\op[1]{{#1^{\operatorname{op}}}} %

\usepackage{xy}
\xyoption{all}

\begin{document}

\title{Semiassociative Algebras over a Field}

\def\UVemail{vishne@math.biu.ac.il}

\author[G.~Blachar]{Guy Blachar}
\address{Department of Mathematics, Bar-Ilan University}
\email{guy.blachar@gmail.com}

\author[D.~Haile]{Darrell Haile}
\address{Department of Mathematics, Indiana University Bloomington}
\email{haile@indiana.edu}

\author[E.~Matzri]{Eliyahu Matzri}
\address{Department of Mathematics, Bar-Ilan University}
\email{elimatzri@gmail.com}

\author[E.~Rein]{Edan Rein}
\address{Department of Mathematics, Bar-Ilan University}
\email{edanrein2000@gmail.com}

\author[U.~Vishne]{Uzi Vishne}
\address{Department of Mathematics, Bar-Ilan University}
\email{\UVemail}

\thanks{}

\renewcommand{\subjclassname}{%
      \textup{2020} Mathematics Subject Classification}
\subjclass{17A60, 16K20, 16K50}

\newcommand\Alg[1]{{\operatorname{Alg}_{#1}}}
\newcommand\Algass[1]{{\operatorname{Alg}^{\operatorname{ass}}_{#1}}}
\newcommand\A[1]{{\mathcal{A}_{#1}}}
\newcommand\T[1]{{\mathcal{T}_{#1}}}

\date{\today}

\begin{abstract}
An associative central simple algebra is a form of matrices, because a maximal \etale{} subalgebra acts on the algebra faithfully by left and right multiplication. In an attempt to extract and isolate the full potential of this point of view, we study nonassociative algebras whose nucleus contains an \etale{} subalgebra bi-acting faithfully on the algebra. These algebras, termed semiassociative, are shown to be the forms of skew matrices, which we are led to define and investigate. Semiassociative algebras modulo skew matrices compose a Brauer monoid, which contains the Brauer group of the field as a unique maximal subgroup.
\end{abstract}

\maketitle

\dedicatory{\it{In memory of Arie-Shmuel (Semi) Olshwanger-Dahari, 1946--2022, a gifted scholar of mathematics and the Hebrew language.
Arie-Shmuel, an academic brother to Vishne and Matzri, devoted his life to learning and teaching.}}

\section{Introduction}

One of the most effective ways to study finite dimensional $F$-central simple algebras is through their maximal subfields. This paper generalizes this approach to the nonassociative realm. Fixing the base field~$F$, our purpose is to study the class of nonassociative central $F$-algebras whose nucleus contains an \etale{} subalgebra satisfying the essential properties which are available in the associative setup. Just as (associative) central simple algebras are the forms of matrices, the algebras in our class turn out to be forms of ``skew matrices'', a deformation of matrices which is amenable to combinatorial techniques.

The study of nonassociative analogs of associative central simple algebras dates back to a work of Dickson \cite{Dic}, who described the nonassociative quaternion algebras, and Albert who described nonassociative crossed products in \cite{Al-II}. Nonassociative quaternion algebras were thoroughly studied by Waterhouse in \cite{Wat}, and generalized by Steele in \cite{St}, defining nonassociative cyclic algebras; these were further generalized in~\cite{BrPu}.
Pumpl\"{u}n and Unger found applications of nonassociative quaternions for space-time block codes \cite{PuU}.
Other works in this field include~\cite{Pu1}, where Pumpl\"{u}n studies Petit algebras \cite{Pet1, Pet2}.

\medskip

We call a nonassociative $F$-central algebra $A$ {\bf{semiassociative}} if its nucleus contains an \etale\ $F$-algebra $K$, such that $A$ is minimally faithful as a $K \tensor K$-module via the action $(k\tensor k')a = kak'$. Associative central simple algebras all belong to this class.
Analogously to the associative theory, the assumption on~$K$ allows us to use the technique of Brauer factor sets as presented by Jacobson \cite[Chapter 2]{Jac} to show that our algebras are forms of ``skew matrices''.

Skew matrix algebras are a twisted version of the standard matrix algebras, defined so that $e_{ij}e_{jk}$ is a scalar multiple of the matrix unit~$e_{ik}$. In general these algebras are far from associative; they are rarely even power-associative. Nevertheless, their nucleus can be directly computed. They have only finitely many ideals, and are often simple. Splitting allows us to deduce properties of arbitrary semiassociative algebras from those of skew matrices.

We define the Brauer monoid to be the class of semiassociative algebras modulo skew matrices. Developing a hierarchy of classes of semiassociative algebras based on properties of the nucleus, we prove that the classical Brauer group is the unique maximal subgroup of the semiassociative Brauer monoid. (Part of the current team developed a theory of associative locally finite infinite dimensional central simple algebras, \cite{infBr}, which gives rise to an ``infinite Brauer monoid'', in whose primary components, once more, the primary components of the classical Brauer group are the unique maximal subgroups; the intersection of the two Brauer monoids is essentially the Brauer group).

If all the products $e_{ij}e_{jk}$ are nonzero, the skew matrix algebra is simple. However, there are simple skew matrix algebras for which some of the products are zero, so if we plan to cover simple semiassociative algebras, we are forced to include this case in our construction. As a result, our skew matrices are diverse enough to serve as forms for Haile's weak crossed product, which were studied in \cite{Haile, Haile2} and in several other papers. Using the separability idempotent, Haile defines a reduced tensor product which closely approximates the second cohomology group of~$F$. In contrast, our approach here is ``$F$-ocentric''\!, taking the product to be induced by the standard tensor product over~$F$. This results in objects of a different flavour.

\medskip

After recalling some basic terminology concerning %
nonassociative algebras, the paper has three parts. Part~\ref{Part1} outlines the fundamental properties of semiassociative algebras. In \Sref{sec:3} we define this class: a finite dimensional nonassociative central $F$-algebra is {\bf{semiassocative}} with respect to an \etale{} $F$-algebra~$K$ of the nucleus $\nuc{A}$, if $A$ is minimally faithful as a module over $K \tensor[F] K$. We prove that this property is independent of $K$. For example, the so-called nonassociative quaternion algebras are semiassociative.
In \Sref{sec:normal!} we discuss the property of normality of an associative simple $F$-algebra whose center is any finite separable extension of $F$, which was introduced in \cite{Haile} as a generalization of normality over a Galois extension of~$F$. This allows us to show in \Sref{sec:Hsec} that semiassociative algebras can be defined in terms of any simple subalgebra of the nucleus whose center is separable over $F$.

\Sref{sec:4} presents skew matrix algebras, which are to semiassociative algebras what standard matrices are to associative central simple algebras. A skew matrix algebra is the algebra spanned by $e_{ij}$, subject to $e_{ij}e_{jk} = c_{ijk}e_{ik}$ for suitable scalars $c_{ijk}$ satisfying $c_{iij} = c_{ijj} = 1$. We show that skew matrix algebras are indeed semiassociative, discuss their presentations, and provide some concrete examples. In \Sref{sec:5} we establish the fact that semiassociative algebras are forms of skew matrices. \Sref{sec:6} deals with Galois descent, and provides more information on a semiassociative algebra and its splitting. The arguments in these two section are transparently based on Chapter 2 of Jacobson's book \cite{Jac}, the key difference being that for associative algebra one begins from the assumption that the algebra is central simple to establish a faithful action of a maximal \etale{} subalgebra, while here the faithful action is our starting point, and indeed semiassociative algebras are not necessarily simple.

Part~\ref{Part2} deals with skew matrices and their combinatorics. We first show, in \Sref{sec:ideals}, that ideals of a skew matrix algebra are spanned by matrix units. We thus define a finite magma whose ideals correspond to the ideals of the skew matrix algebra, and describe an algorithm verifying simplicity. In particular, calling the matrix algebra ``regular'' if all $e_{ij}e_{ji} \neq 0$, we observe that any regular matrix algebra is simple. We then study, in \Srefs{sec:nucprep}{sec:nucofmat}, the nucleus of a skew matrix algebra, which has a Wedderburn decomposition as a direct sum $S + J$, where~$S$ is a maximal regular subalgebra and a direct sum of matrix algebras over~$F$, and~$J$ is the radical. Showing that for any semiassociative algebra the semisimple quotient of the nucleus is separable, we can complete the characterization of forms of skew matrices by proving in \Sref{sec:forms} that they are all semiassociative.

Part~\ref{Part3} revolves around the semiassociative Brauer monoid.
We first observe, in \Sref{sec:tensor}, that the class of semiassociative algebras is closed under tensor product. Anticipating later arguments
which utilize the structure of the nucleus of an arbitrary semiassociative algebra, we also show that the map from a semiassociative algebra to the semisimple quotient of the nucleus is multiplicative. The semiassociative Brauer monoid $\Br[sa](F)$ is defined in \Sref{section:brauer monoid} as the set of equivalence classes of semiassociative algebras, up to tensoring by skew matrices. For simple associative algebras the new equivalence reduces to tensoring by standard matrices, so we have an embedding $\Br(F) \hra \Br[sa](F)$.

The direct summands in the semisimple part of the nucleus, which are associative and simple, even if not central, provide a bridge to the classical Brauer group. We call these direct summands the {\bf{atoms}} of the algebra. A brief discussion of associative semisimple algebras leads us to say that a semiassociative algebra is {\bf{semicentral}} if its atoms are central (\Sref{sec:semicentral}), and {\bf{homogeneous}} if furthermore all the atoms are Brauer equivalent to each other.
The classes of semicentral and homogeneous algebras are closed under tensor product. Altogether, we define in \Sref{sec:monoids} the chain of monoids
$$\xymatrix@C=10pt@R=14pt{
\MM{mat}  & \sub & \MM{hom} &  \sub & \MM{semicentral}  & \sub & \MM{semi}},$$
consisting of (left to right) skew matrices; homogeneous semicentral algebras; semicentral algebras; and all the semiassociative $F$-algebras.
Dividing throughout by skew matrices, we obtain the monoids
$$\Br(F) \sub \Br[sc](F) \sub \Br[sa](F),$$
with the {\bf{semiassociative Brauer monoid}} $\Br[sa](F)$ at the top.
In \Sref{sec:underly} we prove the existence of an underlying (associative) division algebra for homogeneous semiassociative algebras.
Purity of submonoids allows us to show in \Sref{sec:final} that the Brauer group is the unique maximal subgroup of the semiassociative Brauer monoid.

The multiset of atoms of a semiassociative algebra form a nice combinatorial invariant. Since the atoms of a skew matrix algebra are central matrices, the {\emph{distribution}} of atoms is an invariant of the Brauer class. (This distributions for associative central simple algebras are singletons.) We formalize this in \Sref{sec:concrete} by defining a projective map from $\Br[sa](F)$ to the rational ``distribution elements'' in a Grothendieck-like ring of the simple finite dimensional algebras over $F$.

\smallskip

\bigskip

\section{Nonassociative preliminaries}\label{sec:2}

We briefly sketch some notions of nonassociative algebras (\cite{ZSSS} is a standard source).

Recall that for a nonassociative algebra $A$, the {\bf{nucleus}}~$\nuc{A}$ is the intersection of the left, middle and right nuclei, defined as
\begin{eqnarray*}
\nuc[\ell]{A} &= &\set{x \suchthat (x,A,A) = 0}; \\
\nuc[c]{A} &=& \set{x \suchthat (A,x,A) = 0}; \\
\nuc[r]{A} &= &\set{x \suchthat (A,A,x) = 0},
\end{eqnarray*}
where $(x,y,z) = (xy)z-x(yz)$ is the associator. The identity $a(x,y,z)+(a,x,y)z=(ax,y,z)-(a,xy,z)+(a,x,yz)$, which holds in any nonassociative algebra, shows that each of these ``one-sided'' nuclei is an associative subalgebra. The {\bf{center}} of a nonassociative algebra is the set
$\Zent(A) = \set{x \in \nuc{A} \suchthat [x,A]=0}$ of elements in~$\nuc{A}$ commuting with all the elements of $A$. The center is always a commutative and associative subalgebra, and is necessarily a field when~$A$ is simple. An algebra whose center contains a field~$F$ is an $F$-algebra. An $F$-algebra is {\bf{$F$-central}} if the center is equal to~$F$.

The tensor product of two nonassociative algebras $A$ and~$B$ over $F$ is again a nonassociative algebra, with center $\Zent(A\tensor[F] B)=\Zent(A)\tensor[F] \Zent(B)$. The nucleus behaves similarly:

\begin{prop}\label{prop:nuc-tensor}
If $A$ and~$B$ are nonassociative algebras over a field~$F$, then $\nuc{A \tensor B}=\nuc{A}\tensor[F] \nuc{B}$.
\end{prop}
\begin{proof}
Clearly, $\nuc{A}\tensor[F] \nuc{B}\sub \nuc{A\tensor[F] B}$. To see the reverse inclusion, write $x\in \nuc{A\tensor[F] B}$ in the form $x=\sum_{i=1}^n a_i\otimes b_i$ for the minimal possible $n$. Then $a_1,\dots,a_n$ are linearly independent over $F$, and so are $b_1,\dots,b_n$.

Let $a,a'\in A$; then
$$(a,a',x)=(aa')x-a(a'x)=\sum_{i=1}^n(a,a',a_i)\otimes b_i=0.$$
As $b_1,\dots,b_n$ are linearly independent over $F$, we have $(a,a',a_i)=0$ for each~$i$. A similar argument shows that $(a,a_i,a')=(a_i,a,a')=0$. As this holds for all $a,a'\in A$, we have that $a_i\in \nuc{A}$. By symmetry $b_i\in \nuc{B}$, and thus $x\in \nuc{A}\tensor[F] \nuc{B}$.
\end{proof}

\begin{cor}
If $A$ is an $F$-algebra and $F \sub E$ a field extension, then $\nuc{E \tensor A} = E \tensor \nuc{A}$.
\end{cor}

\part{Semiassociative algebras}\label{Part1}

One of the most effective ways to study associative central simple algebras is through their maximal subfields, or more generally, their maximal \etale{} subalgebras. This paper extends this idea to nonassociative $F$-central algebras.
Realizing that the class of all simple nonassociative algebras is too wild to be interesting, we stipulate that the algebra $A$ contains in its nucleus an \etale{} subalgebra~$K$, so that $A$ is a cyclic faithful module over $K^e = K \tensor K$ under the left and right actions. It turns out that this assumption suffice for the construction of a Brauer factor set, as described by Jacobson \cite[Chap~2]{Jac}. Analogously to the classical theory, we show that semiassociative algebras are forms of skew matrices.

\section{Semiassociative algebras}\label{section:semiassociative algebras}\label{sec:3}

Recall that an \etale{} algebra over a field $F$ is a finite direct product of finite separable field extensions of~$F$. This class of commutative algebras is closed under %
tensor products.
Any \etale{} algebra is semisimple. Furthermore if~$K$ is \etale{} over~$F$, then~$K^e$ is semisimple.

A module over any associative ring $R$ is {\bf{minimally faithful}} if it is faithful and has no faithful sections (a section is a quotient of submodules). A semisimple ring has a unique minimally faithful module, namely the direct sum of the simple modules of the summands, with multiplicity one. If the ring $R$ is commutative and semisimple, the unique minimally faithful module is ${}_RR$. In this case a module is minimally faithful if and only if it is cyclic and faithful.

\begin{defn}\label{maindef}
A finite dimensional nonassociative central $F$-algebra is {\bf{semiassociative}} if it is minimally faithful over~$K^e$ for some \etale{} $F$-subalgebra~$K$ of its nucleus.
\end{defn}

\medskip

\begin{exmpl}
Any associative  (finite dimensional) central simple algebra is semiassociative.

Indeed, the algebra has a maximal \etale{} subalgebra $K$ by Koethe's theorem \cite[Cor.~7.1.2$'$]{RowenRT}, and being faithful over $A^e = A \tensor \op{A}$ it is clearly faithful over~$K^e$. (This final argument is not available without associativity.)
\end{exmpl}
\medskip

\begin{rem}[Three-for-two bargain]\label{3for2}
Let $A$ be a nonassociative algebra containing an \etale{} subalgebra~$K$ in its nucleus. Then any two of the following conditions imply the third: $A$ is faithful over $K^e$; $A$ is cyclic over $K^e$; and $\dim A = (\dim K)^2$.
\end{rem}
(This is of course true for any module over an $F$-algebra).
So,

\begin{cor}
The dimension of a semiassociative algebra is a square.
\end{cor}
The square root of the dimension is the {\bf{degree}} of the algebra. %
As in the associative case, we prove below (\Cref{maxcomm}) that $K$ as above is a maximal commutative subalgebra.

\forget
Semiassociativity respects change of basis:
\begin{prop}
Let $F'/F$ be any field extension. If $A$ is semiassociative over $F$ then $F' \tensor[F] A$ is semiassociative over $F'$.
\end{prop}
\forgotten

\subsection{Independence of semiassociativity in~$K$}\label{ss:31}

We say that $A$ is $K$-semiassociative if $A$ and $K$ satisfy the conditions of \Dref{maindef}. We want to show that this property does not depend on~$K$. Let us start with the split case.

\begin{prop}\label{maxi}
Suppose $A$ is an $F^n$-semiassociative algebra. Let $e_1,\dots,e_n$ be an orthogonal system of idempotents generating $F^n$. Then the~$e_i$ are primitive in~$\nuc{A}$.
\end{prop}
(Recall that an idempotent in an associative algebra is primitive if it cannot be nontrivially decomposed as a sum of orthogonal idempotents).
\begin{proof}
By assumption the action of $F^n\tensor F^n$ on $A$ is faithful. In particular each $e_i A e_j$ is nonzero, but since $A$ is the direct sum of these summands, $n^2 = \dim A = \sum_{i,j} \dim e_i A e_j \geq \sum_{i,j} 1 = n^2$, so each $\dim e_i A e_j = 1$. In particular $\dim e_i A e_i = 1$ so $e_i$ is primitive.
\end{proof}

A field extension $E/F$ {\bf{splits}} an \etale\  $F$-algebra $K$ of dimension $n$ if $E\tensor[F] K\cong E^n$.

The nucleus of a nonassociative algebra may contain more than one \etale{} subalgebra of the same dimension.  Let us prove that if $A$ is semiassociative with respect to one of them, then it is semiassociative with respect to all.
\begin{prop}\label{coverall}
Let $A$ be a semiassociative $F$-algebra of degree $n$. Then $A$ is $K'$-semiassociative for any $n$-dimensional \etale{} subalgebra $K' \sub \nuc{A}$.
\end{prop}
\begin{proof}
Suppose $A$ is $K$-semiassociative. Namely, $A$ is faithful under the action of $K \tensor K$.  We need to show that the action of $K' \tensor K'$ is faithful as well. Let~$E$ be a common splitting field of both~$K$ and~$K'$.
Then $\nuc{E \tensor A}$ contains $E \tensor K$ and $E \tensor K'$, which are both isomorphic to $E^n$ as unital algebras, thus each generated by $n$ orthogonal idempotents, $\set{e_1,\dots,e_n}$ and $\set{e_1',\dots,e_n'}$, respectively. By \Pref{maxi}, the system~$\set{e_i}$ is maximal, because $(E \tensor K)^e$ acts faithfully.
Since $\nuc{E \tensor A}$ is finite dimensional, any two maximal systems of orthogonal idempotents are conjugate in~$\nuc{E \tensor A}$, see \cite[Exercise~I.1.12]{AC}. Now, since $\set{e_i'}$ has the same cardinality as $\set{e_i}$ which is maximal, it is maximal as well, and the two systems are conjugate. Namely, $E \tensor K'$ is conjugate to $E \tensor K$ in $E \tensor A$. Since $(E \tensor K)^e$ acts faithfully on $E \tensor A$, and we may conjugate the whole algebra by an element from the nucleus, the action of $(E \tensor K')^e$ is faithful as well. It follows that $K'^e$ acts faithfully on $A$ (because an annihilator would persist after scalar extension). Therefore,~$A$ is~$K'$-semiassociative.
\end{proof}

As discussed above, the main purpose in defining semiassociative algebras is to extend the theory of associative central simple algebras from the point of view of a fixed maximal subfield. However, not every associative algebra in this new class is simple, see \Eref{assocnonsimple} below. We discuss simplicity further in \Sref{sec:ideals}.

\subsection{Examples}

Quite surprisingly, there is no serious discussion in the literature of a nonassocative extension of the structure theory of associative central simple algebras. There are, however, papers on nonassociative crossed products. Albert's paper~\cite{Al-II} is a prime example.
The nonassociative cyclic algebras defined in \cite{St} and generalized in \cite{BrPu} are (in characteristic zero) semiassociative.

\begin{exmpl}
Nonassociative quaternion algebras, which were first defined by \mbox{Dickson}~\cite{Dic} and by Albert~\cite{Al} and later studied and classified by Waterhouse in \cite{Wat}, are semiassociative. A {\bf{nonassociative quaternion algebra}} $Q$ over $F$ is a nonassociative algebra of the form
$$Q = K \oplus Kz$$
with $(k_1+k_2z)(k_1'+k_2'z) = (k_1k_1'+k_2\s(k_2')b)+(k_1k_2'+k_2\s(k_1'))z$,
where $K/F$ is a quadratic Galois extension with $\Gal(K/F)=\sg{\sigma}$ and $b\in \mul{K}$. If $b \in \mul{F}$ this is a standard quaternion algebra. The algebra is $F$-central simple for any $b \in \mul{K}$. The associator $(z,z,z)$ is nonzero when $b \not \in F$.
\end{exmpl}

\begin{exmpl}\label{Q=0}
Similarly, the {\bf{zero quaternion algebra}} $Q = K \oplus Kz$ with the same operation and $z^2 = b = 0$, is an associative semiassociative algebra which is not simple; the semisimple quotient is~$K$.
\end{exmpl}

\begin{exmpl}\label{weak}
Let $K/F$ be a Galois extension of fields. For each $g,h \in G = \Gal(K/F)$ let $c_{g,h} \in K$. Since we allow zeros, this can be viewed as a ``nonassociative weak cocycle'' (``weak'' in the sense of \cite{Haile}). Define the nonassociative crossed product $(K/F,G,c)$ as the algebra $\bigoplus_{g \in G} K z_g$, with $(kz_g)(k'z_{g'}) = kg(k')c_{g,g'}z_{gg'}$. It is easy to see that $K$ is in the nucleus, and $(K/F,G,c)$ is always semiassociative.
\end{exmpl}

\medskip

After introducing skew matrix algebras in \Sref{sec:4}, our aim is to prove that, just as associative central simple algebras are forms of matrices, {\it{semiassociative central simple algebras are forms of skew matrices}}.

\section{Normal algebras}\label{sec:normal!}

This section prepares the ground for \Sref{sec:Hsec}, where normal subalgebras of the nucleus of~$A$ play a key role.
Here we define normality over a separable center, and provide some alternative descriptions. All algebras in this section are associative.
(We restrict the discussion to simple normal algebras; the notion of normality could probably be further extended to semisimple algebras, see \Rref{whatif} below).

\subsection{Decomposing a tensor product}\label{ssec:prenormal}

In order to get a handle on normality, we record a formula for the tensor product over~$F$ of two $L$-central algebras, where~$L$ is a finite separable field extension of~$F$. This is well known when $L/F$ is Galois, but is quite more delicate in the general case.

Let $E$ denote the Galois closure of $L/F$, $G = \Gal(E/F)$ the \mbox{Galois} group, and $H = \Gal(E/L)$ the subgroup corresponding to~$L$. There is a correspondence between right cosets $\tau H$ and embeddings $\tau \co L \ra E$. Likewise there is a correspondence between double cosets $H\tau H \sub G$ and embeddings $\tau \co L \ra E$ up to composition from the left by an automorphism fixing~$L$. Consider the double coset decomposition $G = \bigcup_{i=1}^r H \tau_i H$, with $\tau_1 = 1$. For each $i$, let $L_i$ be the compositum of $L$ and $\tau_i(L)$ in $E$; in particular $L_1 = L$. Equivalently, taking a generator~$\alpha$ of~$L$ over~$F$, we have that $L_i = F[\alpha,\tau_i(\alpha)]$. The group $G$ acts transitively
on the roots of the minimal polynomial, with $H$ being the stabilizer of $\alpha$. Each root is in the $H$-orbit of $\tau_i(\alpha)$ for a unique~$i$.

It is well known that $$L \tensor[F] L \isom L_1 \oplus \cdots \oplus L_r,$$
where the map to the $i$th summand is $\beta \tensor \beta'\mapsto \beta \tau_i(\beta')$.

\begin{lem}\label{firstlem}
Let~$B$ be an $F$-algebra whose center is $L$. Then $$B \tensor[F] L \isom (B \tensor[L] L_1) \oplus \cdots \oplus (B \tensor[L] L_r)$$ under the map taking $b \tensor \beta$ to $(b \tau_i(\beta),\dots,b \tau_r(\beta))$.
\end{lem}

Let $\tau \co L \ra L'$ be any field isomorphism over $F$; if~$B$ is an $L$-algebra, $\tau(B)$ is the algebra obtained by applying $\tau$ to the structure constants. Thus $\tau \co B \ra \tau(B)$ is a formally-defined isomorphism of $F$-algebra, extending $\tau$ from~$L$ to~$B$. But~$\tau(B)$ has no naturally-defined action of~$L$ unless~$L'$ happens to be equal to~$L$, and even then~$\tau(B)$ is not, in general, isomorphic to~$B$ as $L$-algebras.

\begin{lem}\label{howto}
Let $B,B'$ be $F$-algebras whose center is~$L$. Then
$$B \tensor[F] B' \isom (B \tensor[L] L_1 \tensor[\tau_1(L)] \tau_1(B')) \oplus \cdots \oplus (B \tensor[L] L_r \tensor[\tau_r(L)] \tau_r(B'))$$
\end{lem}
\begin{proof}
We have that $B \tensor[F] B' = (B \tensor[F] L) \tensor[L] B'$. By \Lref{firstlem}, the left component can be decomposed as a direct sum of the tensor products $B \tensor[L] L_i$, on each of which the right action of $L$ is twisted by $\tau_i$. Tensoring such a component by $B'$ endowed with the natural action of $L$ from the left, is tantamount to tensoring by $\tau_i(B')$ endowed with the action of $\tau_i(L)$, so we obtain the components $B \tensor[L] L_i \tensor[\tau_i(L)] \tau_i(B')$. More explicitly, the map from $B \tensor[F] B'$ to the $i$th summand is $b \tensor b' \mapsto b \tensor 1 \tensor \tau_i(b')$.
\end{proof}

We thus have, for central simple $L$-algebras~$B$ and $B'$, a decomposition of $B \tensor[F] B'$ as a direct sum of $L_i$-central simple algebras.
\begin{rem}\label{allsame}
The degree of each $B \tensor[L] L_i \tensor[\tau_i(L)] \tau_i(B')$ over the respective center~$L_i$ is equal to the degree of $B\tensor[L] B'$ over $L$, and in particular is independent of~$i$.
\end{rem}

\subsection{Normal algebras}\label{ssec:normality}

The notion of normality for simple algebras was studied in~\cite{EM} (following Jacobson and Teichm\"{u}ller) when the center is Galois over~$F$: the algebra~$B$ is normal if every automorphism of its center extends to the full algebra. Alternatively, if the simple summands of $B \tensor[F] \op{B}$ split.

Now let $L/F$ be any finite separable field extension.
Following~\cite{Haile2}, %
we say that an associative $L$-central simple algebra~$B$ is {\bf{normal over~$F$}} if the simple components of the semisimple algebra $B \tensor[F] \op{B}$ are matrix algebras over their respective centers. For example, $L$ itself is always normal as an algebra, regardless of its Galois-theoretic normality. When $L/F$ is Galois this definition reduces to the one mentioned above, see \cite[Lemma~10.2]{HLS}.

Following \Lref{howto}, let
\begin{equation}\label{Bi}
B_i = B \tensor[L] L_i \tensor[\tau_i(L)] \op{\tau_i(B)}
\end{equation}
be the simple components of $B \tensor[F] \op{B}$. {}From \Rref{allsame} we obtain:
\begin{cor}\label{samefornormal}
An $L$-central simple algebra~$B$ is normal over~$F$ if and only if each~$B_i$ is split over~$L_i$.
When~$B$ is normal, the components of $B \tensor[F] \op{B}$ are $B_i \isom \M[m](L_i)$ where $m = \dimcol{B}{L}$.
\end{cor}

\subsection{Conditions for normality}
{\it{This subsection is not required elsewhere in the paper}.}

A simple algebra $B$ over a Galois extension $E$ of $F$ is normal if the following equivalent conditions hold:
\begin{enumerate}
\item $B \tensor[E] \op{\tau(B)} \sim E$ for any $\tau \in \Gal(E/F)$.
\item $B \isom \tau(B)$ as $E$-algebras for any $\tau \in \Gal(E/F)$.
\item Any automorphism of~$E$ extends to an automorphism of~$B$.
\end{enumerate}

We wish to extend this statement to algebras over a separable center.
Let $L/F$ be a finite separable field extension. Let $\tau \co L \ra L'$ be any field isomorphism over $F$. Let $P = LL'$ denote the compositum within an algebraic closure of $L$.

\def\PB{{P \tensor[L] B}}
\def\PBp{{P \tensor[L'] B'}}
Let $B$ be an $L$-central algebra, and denote $B' = \tau(B)$.  Notice that both $\PB$ and $\PBp$ are $P$-central simple.
\begin{prop}\label{locallynormal}
The following conditions are equivalent.
\begin{enumerate}
\item[(1)] $B \tensor[L] P \tensor[L'] \op{B'} \sim P$ as $P$-algebras.
\item[(2)] $\PBp \isom \PB$ as $P$-algebras.
\item[(2$'$)] $\PBp \isom \PB$ as $L'$-algebras.
\item[(2$''$)] $\PBp \isom \PB$ as $L$-algebras.
\item[(3)] There is an injection $B \ra \PB$ extending $\tau \co L \ra P$.
\item[(3$'$)] There is an injection $B' \ra \PB$ as $L'$-algebras.
\item[(3$''$)] There is an injection $B \ra \PBp$ as $L$-algebras.
\end{enumerate}
\end{prop}
\begin{proof}
$(1) \Longleftrightarrow (2)$: We have that $B \tensor[L] P \tensor[L'] \op{B'} = (B \tensor[L] P) \tensor[P] (P \tensor[L'] \op{B'})$, so this algebra splits over $P$ if and only if $B \tensor[L] P \isom \PBp$ as $P$-algebras.

$(3) \Longleftrightarrow (3')$: Given the injection $\phi' \co B' \ra \PB$ of $L'$-algebras, define $\phi \co B \ra \PB$ by $\phi = \phi'\tau$; and given $\phi$ define $\phi' = \phi \tau^{-1}$.

$(3') \implies (2)$: The natural injection $P \ra \PB$ extends an injection $\phi' \co B' \ra \PB$ as $L'$-algebras to an isomorphism $\PBp \ra P\tensor[L] B$ of $P$-algebras.

$(2') \implies (3')$: Given an isomorphism $f \co \PBp \ra \PB$ over $L'$, define $\phi' \co  B' \ra \PB$ by $\phi'(b') = f(1 \tensor b')$. For $\lam' \in L'$ we have $\phi(\lam') = f(1 \tensor \lam') = f(\lam' \tensor 1) = \lam' \tensor 1$, so $\phi'$ preserves $L'$.

It follows that $(2') \implies (2)$, which seems a bit odd because certainly not every $L'$-isomorphism is an isomorphism over~$P$ (not even for $B = L$). However note that given an $L'$-isomorphism $f \co \PBp \ra \PB$, we derive a $P$-isomorphism  $\hat{f} \co \PBp \ra \PB$ by $\hat{f}(\pi \tensor b) = \pi f(1\tensor b)$.

We proved that $(2)$, $(2')$ and $(3')$ are equivalent. Inverting $\tau$ by switching the roles of $L$ and $L'$ preserves $(2)$ and replaces $(2')$ and $(3')$ by $(2'')$ and $(3'')$, respectively, which completes the proof.
\end{proof}

Let $L$ be a separable field extension of $F$, and $G$ the Galois group of the algebraic closure over $F$.
\begin{cor}\label{whatisnormal}
A simple $L$-central algebra $B$ is normal if and only if one of the conditions of \Pref{locallynormal} holds for every $\tau \in G$.
\end{cor}

If $L$ happens to be Galois over $F$, then $P = L' = L$ for every $\tau$ (with the notation preceding \Pref{locallynormal}), so
\Cref{whatisnormal} reduces to the conditions from the beginning of the subsection.

\section{Subalgebras of the nucleus}\label{sec:Hsec}

Semiassociative algebras are defined in terms of a maximal \etale{} subalgebra of the nucleus, and often in terms of separable maximal subfields. Separable subfields are normal, in the sense defined below. In this section we provide additional characterization of semiassociative algebras, in terms of arbitrary normal simple subalgebras of the nucleus.

\subsection{Quasi-associative algebras}\label{ss:43}

Let $A$ be any nonassociative $F$-algebra. Let~$B$ be a simple subalgebra of the nucleus $\nuc{A}$, with center $L = \Zent(B)$. The left and right actions of~$B$ induce a $B \tensor[F] \op{B}$-module structure on~$A$, and equivalently a homomorphism of $F$-algebras, $$\varphi_{B} \co B \tensor[F] \op{B} \ra \End_{L \tensor[F] L}(A).$$
Clearly, $\varphi_B$ is injective if and only if $A$ is faithful over $B \tensor[F] \op{B}$.

\begin{prop}\label{45}
Let $A$ be a nonassociative $F$-algebra and $B\sub\nuc{A}$ a simple subalgebra with separable center $L=\Zent(B)$. The following are equivalent:
  \begin{enumerate}
  \item
 ~$B$ is normal and~$A$ is minimally faithful as a $B \tensor[F] \op{B}$ module;
  \item The map $\varphi_B \co B\tensor[F]\op{B}\to\End_{L\tensor[F] L}(A)$ is an isomorphism.%
  \end{enumerate}
\end{prop}

\begin{proof}
$(1) \implies (2)$. Assume that~$B$ is normal and~$A$ is minimally faithful as a module over $B \tensor[F] \op{B}$. Write $B\tensor[F]\op{B} = B_1\oplus \cdots \oplus B_r$ where, by \Pref{samefornormal}, the $B_i$ are matrix algebras over their centers~$L_i$.
  Recall that $L\tensor[F] L\cong L_1\oplus\cdots\oplus L_r$.
  In addition, $A$ is minimally faithful over $B\tensor[F]\op{B}$, so there is a module decomposition $A=A_1\oplus\cdots\oplus A_r$, where for each $i$, $A_i$ is the simple module of~$B_i$. But by assumption $B_i\cong \End_{L_i}(A_i)$, so $B\tensor[F]\op{B}= B_1\oplus \cdots \oplus B_r \cong  \End_{L_1}(A_1) \oplus \cdots \oplus \End_{L_r}(A_r) \cong \End_{L\tensor[F] L}(A)$.

$(2) \implies (1)$. Now assume that $B\tensor[F]\op{B}\cong\End_{L\tensor[F] L}(A)$. In particular~$A$ is faithful as a $B\tensor[F]\op{B}$-module. To see the that it is minimally faithful, let $M$ be a proper submodule of $A$. By semisimplicity, $M$ has a complement~$M'$, so that $A=M\oplus M'$. Consider the projection $\pi \co A \ra M'$, which is an endomorphism over $L\tensor[F] L$. Since $\varphi_B$ is surjective there is a nonzero element $y\in B\tensor[F]\op{B}$ which induces~$\pi$, so $yM=0$, and~$M$ is not faithful. This proves that $A$ is minimally faithful as a $B\tensor[F]\op{B}$-module.

We are left with proving that~$B$ is normal over $F$. Write $B\tensor[F]\op{B}=B_1\oplus\cdots\oplus B_r$ and $A=A_1\oplus\cdots\oplus A_r$ as before. Write in addition $B_i\cong\M[n_i](D_i)$ where $D_i$ is an $L_i$-central division algebra, so that $A_i\cong D_i^{n_i}$ is the unique simple $B_i$-module. As a vector space, $A_i \isom L_i^{n_i[D_i:L_i]}$. Then
  \begin{align*}
    \dimcol{\End_{L\tensor[F] L}(A)}{F} & =\sum_{i=1}^r \dimcol{\End_{L_i}(A_i)}{F} = \sum_{i=1}^{r} \dimcol{\M[n_i{[D_i:L_i]}](L_i)}{F} \\ %
    & = \sum_{i=1}^r n_i^2\dimcol{D_i}{L_i}^2\dimcol{L_i}{F}\\
    &\geq \sum_{i=1}^rn_i^2\dimcol{D_i}{L_i}\dimcol{L_i}{F}=\sum_{i=1}^rn_i^2\dimcol{D_i}{F} \\
    &=\sum_{i=1}^r\dimcol{B_i}{F}=[B\tensor[F]\op{B}:F].
  \end{align*}
  But $\End_{L\tensor[F] L}(A) \isom B\tensor[F]\op{B}$ by assumption, forcing an equality of the dimensions, so $D_i=L_i$ for each $i$, and the algebra~$B$ is normal over~$F$.
\end{proof}

\begin{defn}
We say that $A$ is {\bf{quasi-associative with respect to~$B$}} (with separable center $L$)
if it satisfies the equivalent conditions of \Pref{45}.
\end{defn}

\subsection{Double centralizer}

We now prove that quasi-associative algebras satisfy a double centralizer property. Before that, let $A$ be quasi-associative with respect to~$B$, where~$B$ is normal over $F$ with separable center $L$. Recall that the algebra $L\tensor[F] L$ has a distinguished simple component given by the \textbf{separability idempotent} $e$, such that $(L\tensor[F] L)(1-e)$ is the kernel of the multiplication map $L\tensor[F] L\to L$. So in the decomposition $L\tensor[F] L=L_1\oplus\cdots\oplus L_r$, we have that $L_1=(L\tensor[F] L)e=L$. Then $B\tensor[F]\op{B}=B_1\oplus\cdots\oplus B_r$ for $B_1=(B\tensor[F]\op{B})e=B\otimes_L\op{B}$. Writing $A=A_1\oplus\cdots\oplus A_r$ where $A_i$ is a simple $B_i$-module, we deduce that $A_1=eA=B$.

\begin{prop}\label{46}
 Suppose that $A$ is quasi-associative with respect to~$B$, and $L=\Zent(B)$. Then $\Ce[A]{B} = L$ and $\Ce[A]{L} = B$.
\end{prop}

\begin{proof}
Let $e \in L \tensor[F] L$ be the separability idempotent.
  Let $y\in \Ce[A]{L}$, and write $y=ey+(1-e)y$. We claim that $(1-e)y=0$. Indeed, for any $a\in L$ we have $(a\otimes 1-1\otimes a)y=ay-ya = 0$ as $y\in \Ce[A]{L}$. But the elements $a\otimes 1-1\otimes a$ generate the kernel of the multiplication map $L\tensor[F] L\to L$, so $(L\tensor[F] L)(1-e)y=0$. But this forces $(1-e)y=0$, so $y=ey\in eA=B$.

  To prove that $\Ce[A]{B}=L$, note that $\Ce[A]{B}\sub \Ce[A]{L}=B$, and therefore $\Ce[A]{B}=\Ce[B]{B}=\Zent(B)=L$.
\end{proof}

\begin{cor}\label{47}
Let $A$ be any nonassociative $F$-algebra.
 If $A$ is quasi-associative with respect to some simple subalgebra $B\sub\nuc{A}$, then $\Zent(A)=F$.
\end{cor}

\begin{proof}
  If $z\in\Zent(A)$ then $z$ centralizes $L=\Zent(B)$, so $z\in \Ce[A]{L} = B$. But then $z\otimes 1-1\otimes z$ in $B\tensor[F]\op{B}$ annihilates $A$. Hence $z\otimes 1-1\otimes z=0$ (by faithfulness), so $z\in F$.
\end{proof}

\subsection{Quasi-associativity and semiassociativity}

We turn to study the connection between quasi-associative algebras and semiassociative algebras. As we remarked above, when $B=K$ is itself a field, it is automatically a normal algebra over~$F$. Therefore, we have:

\begin{prop}\label{48}
  Let $K/F$ be a separable field extension. Then $A$ is quasi-associative with respect to $K$ if and only if $A$ is $K$-semiassociative.
\end{prop}
\begin{proof}
  As $K$ is already normal over $F$, $A$ is quasi-associative with respect to $K$ if and only if it is minimally faithful as a $K\tensor[F] K$-module, which is the condition for being $K$-semiassociative.
\end{proof}

For the general case, we use the following formula for the dimension:

\begin{lem}\label{49}
  Let $A$ be a quasi-associative algebra with respect to~$B$, and write $L=\Zent(B)$. Then $\dimcol{A}{F}=\dimcol{B}{L}\dimcol{L}{F}^2$.
\end{lem}
\begin{proof}
Apply \Cref{samefornormal} to write $B\tensor[F]\op{B}=B_1\oplus\cdots\oplus B_r$, where $B_i = \M[m](L_i)$ for each $i$, and $m = \dimcol{B}{L}$. The unique minimally faithful module of $B \tensor[F] \op{B}$ is thus the direct sum $L_1^m \oplus \cdots \oplus L_r^m$, and therefore
$$\dimcol{A}{F}= \sum_{i=1}^r\dimcol{L_i^m}{F} = m \sum \dimcol{L_i}{L}\dimcol{L}{F} =m \dimcol{L}{F}^2$$
as required.
\end{proof}

\begin{prop}\label{4-10}
  Let $A$ be a quasi-associative algebra with respect to some %
  $B\sub\nuc{A}$. Then $A$ is semiassociative.
\end{prop}

\begin{proof}
  Write $L=\Zent(B)$. We first assume that~$B$ is a division algebra. Take a maximal separable subfield $K\sub B$, so that $\dimcol{B}{L}=\dimcol{K}{L}^2$. But then $\dimcol{A}{F}=\dimcol{B}{L}\dimcol{L}{F}^2=\dimcol{K}{L}^2\dimcol{L}{F}^2=\dimcol{K}{F}^2$, which is the dimension of the unique minimally faithful $K\tensor[F] K$-module, so $A$ is minimally faithful. That proves that $A$ is quasi-associative with respect to $K$, which is what we want by \Pref{48}.

  For the general case, write $B=\M[k](D)\cong D\tensor[F]\M[k](F)$ for an $L$-central division algebra~$D$. Since $A$ contains $\M[k](F)$ in its nucleus, we can write $A = A'\tensor[F]\M[k](F)$ for a nonassociative algebra $A'$ which can be observed to be quasi-associative with respect to $D$. The first part then shows that $A'$ is semiassociative, so $A\cong A'\tensor[F]\M[k](F)$ is also semiassociative by \Cref{matricesup} below.
\end{proof}

\begin{prop}\label{4-11}
Let $A$ be a semiassociative algebra. Then $A$ is quasi-associative with respect to any normal simple subalgebra $B \sub \nuc{A}$ that has a maximal separable subfield.
\end{prop}
\begin{proof}
Let $K$ be a maximal subfield of~$B$. By \Prefs{coverall}{48},~$A$ is quasi-assocative
with respect to $K$. Let $L$ denote the center of~$B$. Write $L \tensor[F] L = L_1 \oplus \cdots \oplus L_r$ and let $B_i$ be the simple summands of $B \tensor[F] \op{B}$ defined in \eq{Bi}. The annihilator of $A$ as a $B \tensor[F] \op{B}$-module is a direct sum of simple components, but the center $L_i$ of each component is a summand of $L \tensor[F] L \sub K \tensor[F] K$, which acts faithfully by assumption, and thus cannot annihilate~$A$. This proves that $A$ is faithful over $B \tensor[F] \op{B}$.

By \Lref{49} applied to $K$ instead of both~$B$ and $L$, we have that $\dimcol{A}{F} = \dimcol{K}{F}^2 = \dimcol{K}{L}^2\dimcol{L}{F}^2 = \dimcol{B}{L}\dimcol{L}{F}^2$ because $K$ is a maximal subfield of~$B$, but this is the dimension of a minimally faithful module over $B \tensor[F] \op{B}$ by the same lemma. It follows that $A$ is minimally faithful.
\end{proof}

\begin{cor}\label{hurray}
For a nonassociative algebra~$A$ whose nucleus contain a maximal \etale{} subalgebra which is a field, the following are equivalent:
\begin{enumerate}
\item $A$ is semiassociative;
\item $A$ is quasi-associative with respect to some normal simple subalgebra of the nucleus,
\item $A$ is quasi-associative with respect to every normal simple subalgebra of the nucleus containing a maximal subfield.
\end{enumerate}
\end{cor}

\begin{rem}\label{whatif}
By extending the notion of normality to semisimple algebras with \etale{} center, and correspondingly extending the definition of quasi-associativity, we suppose the corollary can be extended to cover any nonassociative algebra~$A$ whose nucleus contains a maximal \etale{} subalgebra. The following would then be equivalent:
\begin{enumerate}
\item $A$ is semiassociative;
\item $A$ is quasi-associative with respect to some normal semisimple subalgebra of the nucleus;
\item $A$ is quasi-associative with respect to every normal semisimple subalgebra of the nucleus.
\end{enumerate}
\end{rem}

\section{Skew matrix algebras}\label{section:skew matrices}\label{sec:4}

Skew matrices are defined by projectively altering the multiplication formula for matrix units.
Let $F$ be a field and $n\geq 1$. Unless otherwise stated, indices will go from $1$ to $n$.
\begin{defn}\label{def:skew-mat}
A \textbf{skew set} of degree $n$ is a tensor $c_{ijk}$ of $n\times n \times n$ scalars from $F$. A skew set~$c$ is \textbf{reduced} if $c_{iij}=c_{jii}=1$ for all $i,j$.

For a reduced skew set $c$, define the \textbf{skew matrix algebra}, denoted $\M[n](F;c)$, to be the $F$-vector space whose basis is the matrix units $e_{ij}$, with multiplication rule
\begin{equation}\label{genmat-prod}
e_{ij}e_{k\ell}=\delta_{jk}c_{ij\ell}e_{i\ell}.
\end{equation}
\end{defn}
With one exception (which will be indicated), all skew sets below are reduced.

For example, the trivial skew set $c_{ijk}=1$ gives the standard matrix algebra, i.e.\ $\M[n](F;1)=\M[n](F)$.
It would have been pleasant to require that all the entries in a skew set be nonzero. Indeed, when all $c_{ijk} \neq 0$, the skew matrix algebra is always simple (\Cref{cijknon0}). However there are simple skew matrix algebras for which many $c_{ijk}$ are zero (see \Eref{simple_with_c=0}), so we are forced to include this possibility in the preset. Consequently, as we see below, some skew matrix algebras are not simple.

\subsection{Basic properties}
\begin{rem}\label{eii}
\begin{enumerate}
\item\label{eii-1} We always have that $e_{ii}e_{ij} = e_{ij}e_{jj} = e_{ij}$.
\item\label{eii-2} The sum $e_{11}+e_{22}+\cdots+e_{nn}$ is the identity element of $\M[n](F;c)$.
\end{enumerate}
In fact, both statements are equivalent to $c$ being reduced, which we always assume to be the case.
\end{rem}

\begin{prop}\label{themat}
The only {\emph{associative simple}} skew matrix algebra is the standard matrix algebra over $F$.
\end{prop}
\begin{proof}
An associative simple algebra of dimension $n^2$ is matrices over a division ring by Wedderburn's theorem, and it has to be $\M[n](F)$ because $e_{11},e_{22},\dots,e_{nn}$ are orthogonal idempotents.
\end{proof}
So the class of associative skew matrix algebras does not include any new simple objects. Moreover,
\begin{prop}\label{nonzeroassoc}
In an associative simple skew matrix algebra $\M[n](F;c)$, all $c_{ijk}$ are nonzero.
\end{prop}
\begin{proof}
By associativity we have that $$e_{ij}Ae_{jk} = \sum_{r,s} F e_{ij}e_{rs}e_{jk} = Fe_{ij}e_{jk} = F c_{ijk}e_{ik},$$
so if $c_{ijk} = 0$ the algebra is not prime.
\end{proof}

In any skew matrix algebra $\M[n](F;c)$, the subalgebra
$$\Delta = F e_{11} + Fe_{22}+\cdots +F e_{nn}$$
will be called {\bf{the diagonal subalgebra}}. Notice that the diagonal subalgebra is always isomorphic to $F \oplus \cdots \oplus F$.
\begin{prop}\label{dn}
The diagonal subalgebra is contained in the nucleus of $\M[n](F;c)$, for any skew set $c$.
\end{prop}

\begin{proof}
An associator of matrix units $(e_{ii'},e_{jj'},e_{kk'})$ is zero regardless of~$(c)$, unless $i' = j$ and $j' = k$. Now, by \Rref{eii}.\eq{eii-1}, $(e_{ii}e_{ij})e_{jk} = e_{ij}e_{jk} = e_{ii}(e_{ij}e_{jk})$;
$(e_{ki}e_{ii})e_{ij} = e_{ki}e_{ij} = e_{ki}(e_{ii}e_{ij})$ and $(e_{jk}e_{ki})e_{ii} = c_{jki}e_{ji}e_{ii} = e_{jk}e_{ki}$. This shows that $e_{ii}$ is in the nucleus.
\end{proof}

\begin{prop}\label{centeris}
The center of $\M[n](F;c)$ contains only the scalar matrices.
\end{prop}
\begin{proof}
The proof for the standard matrix algebra works verbatim for every skew matrix algebra.
\forget %
Let $a = \sum \alpha_{ij}e_{ij}$. Then $[e_{kk},a] = \sum_{ij} \alpha_{ij}(e_{kk}e_{ij}-e_{ij}e_{kk}) =
\sum_{j} \alpha_{kj}e_{kj} - \sum_{i} \alpha_{ik}e_{ik}$. If $a$ is central then $\alpha_{kj} = 0$ for every $k \neq j$, so $a$ is diagonal. But then, for fixed $k,k'$, $[e_{kk'},a] = \sum_i \alpha_{ii} [e_{kk'}, e_{ii}] = (\alpha_{k'k'} - \alpha_{kk}) e_{kk'}$, so $\alpha_{kk} = \alpha_{k'k'}$.
\forgotten
\end{proof}

\begin{prop}\label{part1go}
Every skew matrix algebra is a semiassociative $F$-central algebra (with respect to the diagonal subalgebra).
\end{prop}
\begin{proof}
Following \Prefs{dn}{centeris}, it remains to show that the annihilator of $A$ as an $\Delta \tensor \Delta$-module is trivial; but every ideal of $\Delta \tensor \Delta$ contains some $e_{ii} \tensor e_{jj}$, and $(e_{ii} \tensor e_{jj})\cdot e_{ij} = e_{ii}e_{ij}e_{jj} = e_{ij} \neq 0$.
\end{proof}

\begin{prop}\label{centralizerofD}
The diagonal subalgebra is self-centralizing.
\end{prop}
\begin{proof}
As in \Pref{centeris}, because $e_{ii}e_{ij} = e_{ij} = e_{ij}e_{jj}$.
\end{proof}

\begin{rem}\label{tensormat}
The tensor product of skew matrix algebras $\M[n](F;c)$ and $\M[n'](F;c')$ is a skew matrix algebra $\M[nn'](F;c\tensor c')$, with the skew set $(c\tensor c')_{ii',jj',kk'} = c_{ijk}c'_{i'j'k'}$.
\end{rem}

Comparing to \Tref{basicTx}, this remark proves that skew matrix algebras compose a submonoid of the monoid of all semiassociative algebras.

\subsection{Equivalence of skew sets}
We say that two reduced skew sets $c,c'$ are {\bf{equivalent}}, denoted $c \sim c'$, if there are nonzero scalars $\gamma_{ij}$, with $\gamma_{ii} = 1$, such that
$$c'_{ijk} = \gamma_{ij}\gamma_{jk}\gamma_{ik}^{-1} \, c_{ijk}.$$
This is indeed an equivalence relation.

We say that an isomorphism $\M[n](F;c) \ra \M[n](F;c')$ is {\bf{homogeneous}} if it preserves each of the linear subspaces $Fe_{ij}$. An equivalence induces a homogeneous isomorphism defined by $e_{ij} \mapsto \gamma_{ij}e_{ij}$.

\begin{rem}
One can view equivalence classes as orbits under the action of reduced matrices $(\gamma_{ij})$ (with $\gamma_{ii}=1$), which form a group with respect to pointwise multiplication, in the obvious manner. The matrices $\gamma_{ij} = \epsilon_i \epsilon_j^{-1}$ (for $\epsilon_1,\dots,\epsilon_n \in \mul{F}$) are the ones acting trivially, %
so the dimension of a generic orbit is $n(n-1)-(n-1) = (n-1)^2$.
\end{rem}

Note that conjugation by the diagonal matrix $\diag\set{\epsilon_1,\dots,\epsilon_n}$ is a homogeneous automorphism, leaving $c$ unchanged.

\medskip
Assume $\M[n](F;c)$ is associative. In \Pref{nonzeroassoc} we saw that if the algebra is simple, then all $c_{ijk} \neq 0$. Let us now prove the converse (this is a special case of \Cref{cijknon0} below).
\begin{exmpl}\label{assocskew-3}
Assume $\M[n](F;c)$ is associative and all $c_{ijk} \neq 0$. Take $\gamma_{ij} = c_{1ij}$. Then $c_{jk\ell} = \gamma_{jk}\gamma_{k\ell}\gamma_{j\ell}^{-1}$ because $(\gamma_{jk}\gamma_{k\ell} - c_{jk\ell}\gamma_{j\ell})e_{1\ell} = (e_{1j},e_{jk},e_{k\ell})$ is zero when the algebra is associative.
So $c \sim 1$, and $\M[n](F;c) \isom \M[n](F;1) = \M[n](F)$.
\end{exmpl}

\subsection{Examples}

As an illustration of the diversity of skew matrix algebras, let us describe the $2$-by-$2$ skew matrix algebras.
\begin{exmpl}[Skew matrices of degree $n = 2$]\label{n=2}
There are three types of skew matrix algebras of degree $2$:
\begin{enumerate}
\item\label{n=2.1} An associative nonsimple algebra. %
\item\label{n=2.2} A nonsimple algebra which is not associative. %
\item\label{n=2.3} A one-parameter family of simple algebras, whose only associative member is the standard matrix algebra.
\end{enumerate}
Case \eq{n=2.3} are the split nonassociative quaternion algebras of Waterhouse~\cite{Wat}.

To check this statement, notice that a reduced skew set for $n = 2$ has only two nontrivial entries, namely $c_{121}$ and $c_{212}$. The action of a reduced matrix $(\gamma_{ij})$ multiplies both parameters by $\gamma_{12}\gamma_{21}$. The nontrivial associators are $(e_{12},e_{21},e_{12}) = (c_{121}-c_{212})e_{12}$ and $(e_{21},e_{12},e_{21}) = -(c_{121}-c_{212})e_{21}$, so $\M[2](F;c)$ is associative if and only if $c_{121}=c_{212}$.
\begin{enumerate}
\item If $c_{121} = c_{212} = 0$ then $J = Fe_{12}+Fe_{21}$ is the radical; this is case~\eq{n=2.1}.
\item If $c_{121} = 0$ and $c_{212} \neq 0$, we may assume $c_{212} = 1$; then the associator ideal contains $e_{12}$ and $e_{21}$, and $E = Fe_{12}+Fe_{21}+Fe_{22}$ is an idempotent ideal. The dual case gives an isomorphic algebra, by transposing the indices. This is case \eq{n=2.2}.
\item Finally, if $c_{121},c_{212} \neq 0$, the ratio parameter $c_{121}c_{212}^{-1}$ determines the algebra; we will see below that all such algebras are simple. The ratio is $1$ for the standard matrix algebra, and this family is case \eq{n=2.3}.
    \end{enumerate}
\end{exmpl}

In particular,
\begin{exmpl}[Associative skew matrix algebra which is not simple]\label{assocnonsimple}
The {\bf{zero matrix algebra}} algebra $\M[2](F;c)$, for $c_{121}=c_{212}=0$, is associative but not simple. (This is a zero quaternion algebra in the sense of \Eref{Q=0}, with $K = F \oplus F$).
\end{exmpl}

\begin{rem}\label{1gen}%
If $c_{121} \neq c_{212}$ then $\M[2](F;c)$ is generated by a single element. Take $z = e_{12}+e_{21}$. Then $z^2 = c_{121}e_{11}+c_{212}e_{22}$ so $\operatorname{span}\set{1,z^2} = \operatorname{span}\set{e_{11},e_{22}}$,
while $(z^2)z = c_{121}e_{12}+c_{212}e_{21}$,
so $\operatorname{span}\set{z,(z^2)z} = \operatorname{span}\set{e_{12},e_{21}}$. Moreover, $(z,z,z) = [z^2,z] = (c_{121}-c_{212})(e_{12}-e_{21})$, so the algebra is not power-associative.
\end{rem}

\section{Matrix forms}\label{sec:5}

We say that a semiassociative algebra is {\bf{split}} if it is a skew matrix algebra.
A field extension $E/F$ {\bf{splits}} a semiassociative algebra~$A$, if $A_E = E \tensor A$ is split.

\begin{thm}\label{main-thm}
Let $F$ be an infinite field and $A$ a semiassociative $F$-algebra of degree $n$. A field splitting an $n$-dimensional \etale{} subalgebra $K \sub \nuc{A}$ splits $A$ as well.
Furthermore, if~$E$ splits~$K$ then $E \tensor K$ can be assumed to be the diagonal subalgebra of $E\tensor A$.
\end{thm}
The proof is given in \Ssref{ssec:52}.

\subsection{Applications}
We can now obtain a useful nonassociative analogue of the associative splitting criterion (see \eg\ in \cite[Proposition 2.2.8]{GS}):
\begin{cor}\label{split-nuc-skew}
A semiassociative algebra of degree $n$ is split if and only if $F^n$ is contained in the nucleus (as a unital subalgebra).
\end{cor}
\begin{proof}
$(\Rightarrow)$: A skew matrix algebra contains a copy of $F^n$ by \Pref{dn}.
$(\Leftarrow)$: Since $A$ is semiassociative and $F^n \sub \nuc{A}$, $A$ is $F^n$-semiassociative, so we are done by \Tref{main-thm} by taking $E = F$.
\end{proof}
\forget
\begin{cor}\label{cor:split-nuc-skew}
An $F$-central algebra of dimension $n^2$ is a skew matrix algebra if and only if it is $F^n$-semiassociative (where $F^n$ is a unital subalgebra).
\end{cor}
\begin{proof}
$(\Rightarrow)$ follows from \Pref{dn}.
$(\Leftarrow)$ by \Tref{main-thm}, taking $E = F$.
\end{proof}
\forgotten

Splitting gives us another essential property of $K$.
\begin{cor}\label{maxcomm}
If $A$ is semiassociative of degree~$n$, then any $n$-dimensional \etale{} subalgebra $K \sub \nuc{A}$ is a maximal commutative subalgebra.
\end{cor}
\begin{proof}
Let $E$ be a splitting field of $K$. By \Tref{main-thm}, $E\tensor A \isom \M[n](E;c)$ for a suitable reduced skew set $c_{ijk} \in E$. Let $K \sub K' \sub A$ be a commutative subalgebra. Then $E\tensor K \sub E\tensor K' \sub \M[n](E;c)$, but $E \tensor K \isom E^n$, so $E \tensor K'$ contains the diagonal subalgebra of $\M[n](E;c)$, which is its own centralizer by \Pref{centralizerofD}. It follows that $E \tensor K' = E \tensor K$, so $K' = K$.
\end{proof}

\begin{rem}
For a semiassociative algebra $A$, the center of the nucleus satisfies $\Zent(\nuc{A}) \sub \bigcap K$ where the intersection is over all $n$-dimensional \etale{} subalgebras of $\nuc{A}$; indeed for any such algebra $K$, we have
$$F \sub \Zent(\nuc{A}) \sub K \sub \nuc{A} \sub A.$$

The nontrivial claim here is that $\Zent(\nuc{A}) \sub K$. But an element $z$ in the center of $\nuc{A}$ must centralize~$K$, so $z \in K$ by \Pref{maxcomm}.
\end{rem}

Here is a case where being nonassociative is an advantage:
\begin{cor}\label{genericsplitting}
Let $A$ be a $K$-semiassociative algebra for which $K = \nuc{A}$.
Then a field extension $E/F$ splits $A$ if and only if it splits~$K$.
\end{cor}
\begin{proof}
$(\Leftarrow)$: This is \Tref{main-thm}.

$(\Rightarrow)$: By assumption $K = \nuc{A}$, so $E \tensor K = E \tensor \nuc{A} = \nuc{E \tensor A} = \nuc{\M[n](E;c)}$, which contains the diagonal subalgebra $E^n$ and is thus equal to it. Namely $E \tensor K = E^n$, so $E$ splits $K$.
\end{proof}
In comparison, a field may split an associative central simple algebra, without splitting any of its maximal subfields; the diagonal of the split algebra does not necessarily descend to a subfield in the original algebra (for example a biquadratic extension can split a quaternion algebra without splitting its subfields; a maximal subfield of a noncrossed product splits the algebra but not its subfields).

\subsection{Proof of \Tref{main-thm}}\label{ssec:52}

Let $F$ be an infinite field, let $K$ be an \etale\ $F$-algebra of degree $n$, and let $A$ be a $K$-semiassociative $F$-algebra. Let $E$ be a splitting field of $K$. The proof is an expansion on the ideas of \Ssref{ss:31}.

Fix an element $v \in A$ for which $A = KvK$. This element exists since $A$ is assumed a cyclic module.
\begin{rem}
For central simple associative algebras, where $K$ is only assumed to be a maximal subfield, there are many proofs for cyclicity in the literature, see \cite{Albert1} or \cite[Appendix]{KaK} and the references therein. The proof in \cite[Theorem~2.2.2]{Jac}, using properties of modules over Frobenius algebras and the assumption on dimensions, applies in our setup, and is essentially \Rref{3for2}.
\end{rem}

As $F$ is infinite, $K$ is generated by a single element \cite[Cor.~4.2(d)]{FR17}, say $K = F[u]$.

Thus we may write $K \cong F[x]/\sg{f_u(x)}$, where $f_u$, the minimal polynomial of $u$ over $F$ (which may be reducible), has~$n$ distinct roots $r_1,\dots,r_n$ in the splitting field~$E$.

\begin{cor}\label{cor:f-basis}
$\set{u^k vu^l}_{k,l=0}^{n-1}$ is an $F$-basis for $A$.
\end{cor}

Consider the algebra $A_E=E\tensor[F] A$, whose center is $E = E \tensor F$. Likewise we denote $K_E = E \otimes K$.
The following is well-known:
\begin{rem}\label{fact:idempotents}
$K_E=E\tensor[F] K=\sum_{i=1}^n Ee_i$ where $e_1,\dots,e_n \in E \tensor[F] K$ is a set of orthogonal idempotents summing to~$1$. In particular $K_E e_i = Ee_i$.
\end{rem}

Working in $A_E = E \tensor A$, let $$v_{ij}=e_i(1\otimes v)e_j.$$

\begin{prop}\label{prop:vij-base}
The set $\set{v_{ij}}$ forms an $E$-basis for $A_E$.
\end{prop}
\begin{proof}
As $\sum_{i=1}^n e_i=1$ in $K_E$, we have
\begin{eqnarray*}
  A_E & = & \sum_{i,j=1}^n e_i A_E e_j=\sum_{i,j=1}^n e_i (E \tensor KvK)e_j  \\
&  = & \sum_{i,j=1}^n e_i (E \tensor K)(1\tensor v)(E \tensor K)e_j = \sum_{i,j=1}^n e_i K_E (1 \tensor v) K_E e_j \\
&  = & \sum_{i,j=1}^nEe_i (1 \tensor v) e_j E
=\sum_{i,j=1}^nEv_{ij}E=\sum_{i,j=1}^nEv_{ij}.
\end{eqnarray*}
Since each summand in the right-hand side has dimension at most~$1$, the $v_{ij}$ must be nonzero, and the claim follows.
\end{proof}

\begin{prop}\label{prop:vij-mult}
There are (unique) scalars $\set{c_{ijk}}_{i,j,k=1}^n$ in $E$ such that
$$v_{ij}v_{kl}=\delta_{jk}c_{ijl}v_{il}.$$
\end{prop}
\begin{proof}
Consider the product of $v_{ij}$ and $v_{kl}$. If $j\neq k$, we have
$$v_{ij}v_{kl}=e_i(1\otimes v)e_je_k(1\otimes v)e_l=0.$$ If $j=k$,
then $v_{ij}v_{jl}=e_i(1\otimes v)e_je_j(1\otimes v)e_l\sub e_iA_Ee_l=Ev_{il}$.
\end{proof}

To complete the proof, we need to replace the skew set $(c_{ijk})$ by a reduced skew set.

\begin{prop}\label{prop:vij-unit}
Write $1=\sum_{i,j=1}^n \alpha_{ij}v_{ij}\in A_E$. Then:
\begin{enumerate}
  \item\label{item:unit-1} $\alpha_{ij}=0$ for $i \neq j$;
  \item\label{item:unit-2} $e_i = \alpha_{ii}v_{ii}$ and $\alpha_{ii}=c_{iii}^{-1}=c_{iij}^{-1}=c_{jii}^{-1}$.
\end{enumerate}
\end{prop}
\begin{proof}
As $1=\sum_{i,j=1}^n\alpha_{ij}v_{ij}$, we have
$$\delta_{kl}e_k=e_ke_l=e_k1e_l=\sum_{i,j=1}^n\alpha_{ij}e_kv_{ij}e_l=\alpha_{kl}v_{kl},$$
using $e_k v_{ij} e_l = \delta_{ik}\delta_{jl} v_{ij}$.
Therefore, for $k\neq l$ we have $\alpha_{kl}=0$ since $v_{kl} \neq 0$. When $l = k$ we proved $\alpha_{kk}v_{kk}=e_k$. For any $j$, we now see that
\begin{align*}
  v_{kj} & =1v_{kj}=\sum_{i=1}^n\alpha_{ii}v_{ii}v_{kj}=\alpha_{kk}c_{kkj}v_{kj}\\
  v_{jk} & =v_{jk}1=\sum_{i=1}^n\alpha_{ii}v_{jk}v_{ii}=\alpha_{kk}c_{jkk}v_{jk},
\end{align*}
so $c_{kkj},c_{jkk}$ are nonzero and $\alpha_{kk}=c_{kkj}^{-1}=c_{jkk}^{-1}$, proving \eqref{item:unit-2}.
\end{proof}

\begin{proof}[Proof of \Tref{main-thm}]
Let $E$ be a field splitting $f_u$, as considered above.
\Pref{prop:vij-unit} shows that $c_{iii}\in\mul{E}$, so we can take
\begin{equation}\label{eij}
e_{ij}=c_{iii}^{-1}v_{ij},
\end{equation}
and then $\set{e_{ij}}_{i,j=1}^n$ forms an $E$-basis for $A_E$ by \Pref{prop:vij-base}. Also, $e_{ii} = \alpha_{ii}v_{ii} = e_i$
by \Pref{prop:vij-unit}, and so $\sum_{i=1}^ne_{ii}= \sum_i e_i = 1$, so it remains to check the defining multiplication formula. By \Pref{prop:vij-mult},
$$e_{ij}e_{kl}=c_{iii}^{-1}c_{kkk}^{-1}v_{ij}v_{kl}=\delta_{jk}c_{iii}^{-1}c_{jjj}^{-1}c_{ijl}v_{il}= \delta_{jk}c_{jjj}^{-1}c_{ijl}e_{il};$$
therefore, we see that the skew set corresponding to $\set{e_{ij}}$ is given by $c'_{ijk}=c_{jjj}^{-1}c_{ijk}$. Clearly
$c'_{iij} = c_{iii}^{-1}c_{iij} = 1$ and $c'_{ijj} = c_{jjj}^{-1}c_{ijj} = 1$, so $c'$ is reduced. Finally, $A_E\cong\M[n](E;c')$.
\end{proof}

\section{The action of $\Gal(E/F)$ on $A_E$}\label{section:gal-action}\label{sec:6}

In \Tref{main-thm} we take $E$ to be any field splitting the algebra~$K$. In this section we further assume that $E/F$ is Galois. This will pose further restrictions on the factor set of the skew matrix algebra $A_E \isom \M[n](E;c)$.

Recall the setup of \Tref{main-thm}. Let $F$ be an infinite field, let~$K$ be an \etale\  $F$-algebra of degree $n$, and suppose that $A$ is a $K$-semiassociative $F$-algebra. Write $K=F[u]$ for some $u$ for which the minimal polynomial $f_u$ has $n$ distinct roots $r_1,\dots,r_n$ in its splitting field $E$. Choose $v\in A$ such that $A=KvK$, and let $v_{ij}=e_i(1\otimes v)e_j\in A_E=E\tensor[F] A$ where $e_1,\dots,e_n$ is a set of orthogonal idempotents in $K_E=E\tensor[F] K$.

\smallskip

Let $G = \Gal(E/F)$ be the Galois group of $E/F$. The action of $G = \Gal(E/F)$ extends to $A_E$ by  $\sigma(e\otimes a)=\sigma(e)\otimes a$ for any $\s \in G$. Identifying $A$ with $F \tensor A \sub E \tensor A = A_E$, it follows that
\begin{equation}\label{fixed-AE} %
(A_E)^G = A.
\end{equation}

\begin{cor}
Every semiassociative algebra is the invariant subalgebra of a skew matrix algebra, under a suitable finite group action.
\end{cor}
\begin{proof}
Take $E$ to be the splitting field of~$K$.
\end{proof}

\begin{rem}[{\cite[Subsection 2.3.1]{Jac}}]\label{ei=}
The idempotents $e_1,\dots,e_n$ can be expressed in terms of $r_1,\dots,r_n$:
\begin{equation}\label{eq:ei-expr}
e_i=\frac{(u-r_1)\cdots(u-r_{i-1})(u-r_{i+1})\cdots(u-r_n)}{(r_i-r_1)\cdots(r_i-r_{i-1})(r_i-r_{i+1})\cdots(r_i-r_n)}.
\end{equation}
As $(u-r_i)e_i=0$, we have $ue_i=r_ie_i$.
\end{rem}

In addition, each $\sigma\in G$ induces a permutation on $\set{1,\dots,n}$, by  $\sigma(r_i)=r_{\sigma(i)}$. From \Rref{ei=} we obtain that $\s(e_i) = e_{\s(i)}$.

\begin{prop}\label{prop:sigma-vij}
The action of $G$ on the $v_{ij}$ is by $\sigma(v_{ij})=v_{\sigma(i)\sigma(j)}$.
\end{prop}
\begin{proof}
Recall that $v_{ij}=e_i(1\otimes v)e_j$, so
$$\sigma(v_{ij})=\sigma(e_i)\sigma(1\otimes v)\sigma(e_j)=e_{\s(i)}(1\otimes v)e_{\s(j)} = v_{\s(i),\s(j)}.$$
\end{proof}

\begin{cor}[Conjugacy condition on the factor set]
Let $c$ be the factor set obtained above, for which $A_E \isom \M[n](E;c)$. Then $c$ satisfies the conjugacy condition:
\begin{equation}\label{Gact}
\sigma(c_{ijk})=c_{\sigma(i)\sigma(j)\sigma(k)}
\end{equation}
for every $\s \in G$.
\end{cor}
\begin{proof}
Let $\sigma\in G$. On the one hand, by \Pref{prop:vij-mult},
$$v_{\sigma(i)\sigma(j)}v_{\sigma(j)\sigma(k)}=c_{\sigma(i)\sigma(j)\sigma(k)}v_{\sigma(i)\sigma(k)}.$$
On the other hand, %
using \Pref{prop:sigma-vij},
$$v_{\sigma(i)\sigma(j)}v_{\sigma(j)\sigma(k)}=\sigma(v_{ij}v_{jk})=\sigma(c_{ijk}v_{ik})=\sigma(c_{ijk})v_{\sigma(i)\sigma(k)}.$$
\end{proof}

\begin{cor}\label{actone}
The action of $G$ on $\M[n](E;c)=\sum Ee_{ij}$ is by $$\s(\sum \alpha_{ij}e_{ij}) = \sum \s(\alpha_{ij})e_{\s(i)\s(j)}.$$
\end{cor}
\begin{proof}
We have that $\s(e_{ij}) = e_{\s(i)\s(j)}$ by \eq{eij}, \Pref{prop:sigma-vij} and \eq{Gact}.
\end{proof}

\forget %
\begin{cor}\label{cor:conj-elem}
Let $x=\sum_{i,j=1}^n\alpha_{ij}v_{ij}\in A_E$. Then $x\in A$ if and only if $x$ satisfies the ``conjugacy condition'':
$$\sigma(\alpha_{ij})=\alpha_{\sigma(i)\sigma(j)}$$
for every $\sigma \in G$.
\end{cor}
\begin{proof}
By \Pref{prop:sigma-vij}, $\sigma(x)=\sum_{i,j=1}^n\sigma(\alpha_{ij})v_{\sigma(i)\sigma(j)}$. By \eq{fixed-AE}, $x\in A$ if and only if $\sigma(x)=x$ for all $\sigma\in G$, if and only if $\sigma(\alpha_{ij})=\alpha_{\sigma(i)\sigma(j)}$ for all $\sigma\in G$.
\end{proof}
\forgotten

\begin{cor}\label{cor:conj-elem}
Let $x=\sum_{i,j=1}^n\alpha_{ij}e_{ij}\in A_E$. Then $x\in A$ if and only if $x$ satisfies the ``conjugacy condition'':
$$\sigma(\alpha_{ij})=\alpha_{\sigma(i)\sigma(j)}$$
for every $\sigma \in G$.
\end{cor}
\begin{proof}
By \Pref{actone}, $\sigma(x)=\sum_{i,j=1}^n\sigma(\alpha_{ij})e_{\sigma(i)\sigma(j)}$. By \eq{fixed-AE}, $x\in A$ if and only if $\sigma(x)=x$ for all $\sigma\in G$, if and only if $\sigma(\alpha_{ij})=\alpha_{\sigma(i)\sigma(j)}$ for all $\sigma\in G$.
\end{proof}

\begin{rem}\label{rem:choose-reduced}
The element $v$ for which $A = KvK$ can be chosen so that the associated factor set is reduced. More explicitly, in the proof of \Tref{main-thm} we used $v_{ij}=e_ive_j$ to find a new $E$-basis for $A_E$, $e_{ij}=c_{jjj}^{-1}v_{ij}$, which showed that $A_E$ is isomorphic to a skew matrix algebra. We note that this change of basis can be regarded as a change of the generating element $v$ (taken so that $A=KvK$). Indeed, let $w=\sum_{i,j=1}^ne_{ij}$. \Cref{cor:conj-elem} shows that $w\in A$. Now $e_iwe_j=e_{ij}$, so taking~$w$ as our initial generator would have yielded the $E$-basis~$\set{e_{ij}}$.
\end{rem}

Reversing the arguments, we proved the following.
\begin{cor}
Let $E/F$ be a Galois extension, and suppose that the skew matrix $E$-algebra $M$ is defined over $F$; namely $M \isom E \tensor A$ for a semiassociative $F$-central algebra $A$. Then $M \isom \M[n](E;c)$ for a skew set $c_{ijk} \in E$ satisfying~\eq{Gact}.
\end{cor}

\part{Skew matrix algebras and the nucleus}\label{Part2}

The key example for semiassociative algebras are skew matrices.
In this part we define and study this class of algebras.

\section{Ideals of skew matrices}\label{sec:ideals}

Although our main interest is in simple skew matrix algebras, we devote this section to ideals in the non-simple case. The key point is that (thanks to the role of the idempotents $e_{ii}$) ideals are always homogeneous, reducing properties of ideals of skew matrices to combinatorial arguments.

\subsection{The grading}
Let $\E_n$ denote the semigroup of order $n^2+1$ whose elements are the matrix units $e_{ij}$ and zero, with the multiplication rule induced from standard matrices. When grading an algebra by a semigroup with zero, we tacitly assume that the homogeneous component of the zero element is the zero space.
\begin{rem}\label{gradeEc}
Every skew matrix algebra $\M[n](F;c)$ is graded by $\E_n$; the homogeneous component of $e_{ij} \in \E_n$ is the subspace $Fe_{ij}$ of the algebra.
\end{rem}
This is a {\bf{fine grading}}, in the sense that the dimension of every homogeneous component of nonzero degree is~$1$.

\subsection{Ideals are homogeneous}

Recall that an ideal of a graded algebra is {\bf{homogenous}} if it is generated by its homogeneous elements.

We say that a semigroup with zero $S$ is {\bf{separated}} if for every $s',s'' \in S$, the set $s'Ss''$ has at most one nonzero element.
\begin{thm}\label{sep}
Let $S$ be a separated semigroup.
Then every ideal of a unital $S$-graded (nonassociative) algebra is homogeneous.
\end{thm}
\begin{proof}
Let $A$ be a unital $S$-graded algebra. Decompose the identity $1 = \sum u_s$ into homogeneous components. For $s',s''$, let $s \in S$ be the unique element such that $s'Ss'' \cup \set{0} = \set{s,0}$ (possibly $s =0$). Ranging over homogeneous elements, we have that $(u_{s'}A)u_{s''} \sub A_s+A_0 = A_s$.

Now let $I \normali A$ be an ideal, and let $a \in I$. Clearly, $a = (1 a) 1 = \sum_{s',s''} (u_{s'}a)u_{s''}$, but as we have just seen, the summands, which are all in $I$, are homogeneous.
It follows that
every element of~$I$ decomposes as a sum of homogenous elements of~$I$.
\end{proof}

\begin{prop}
The semigroup $\E_n$ is separated.
\end{prop}
\begin{proof}
Indeed, $e_{ij}\E_ne_{k\ell} = \set{e_{i\ell},0}$.
\end{proof}

\begin{cor}
Every ideal of a skew matrix algebra is homogeneous (with respect to the $\E_n$-grading).
\end{cor}

The next step is to use the grading to control ideals of skew matrix algebras. However, the grading by $\E_n$ is not restrictive enough, and we need to introduce grading by other magmas.

\subsection{The frame}

To every skew matrix algebra we associate a finite magma (namely a set with binary operation) with zero, as follows. Let~$c$ be a skew set. The {\bf{frame}} of $c$ is the magma $\E(c)$, of cardinality $n^2+1$, whose elements are the linear subspaces $Fe_{ij}$ and the zero space, with the multiplication induced by \eq{genmat-prod}, namely, $Fe_{ij}\cdot Fe_{jk} = Fc_{ijk}e_{ik}$ and $Fe_{ij} \cdot Fe_{k\ell} = 0$ for $j \neq k$. Clearly, the frame only depends on whether or not each $c_{ijk}$ is zero. Notice that $\E_n$ is the frame of the standard matrices (matching $e_{ij}$ with $Fe_{ij}$).

Multiplication by $Fe_{ii}$ is determined by \Pref{eii}.\eq{eii-1}, but other than that, $(Fe_{ij})(Fe_{jk})$ can be either $Fe_{ik}$ or $0$, and any combination is a possible frame. There are $2^{n(n-1)^2}$ frames for skew matrices of order~$n$. (The four frames for $n=2$ implicitly appear in \Rref{n=2}).

Frames are not necessarily associative:
\begin{exmpl}
The frames in \Eref{n=2}.\eq{n=2.2} and in \Eref{simple_with_c=0} are not associative.
\end{exmpl}

A fine grading $A = \bigoplus A_s$, by a magma $S$, is {\bf{strong}} if $A_s A_{s'} = A_{ss'}$ for any $s,s'$ in the grading magma (this term is being used when the grading is not necessarily fine).

The motivation to define frames is, of course:
\begin{rem}
Every skew matrix algebra $\M[n](F;c)$ is strongly graded by its frame.
\end{rem}
One could view the skew matrix algebra as a ``skew group algebra'' of its frame, or rather a skew magma algebra, but we refrain from using this terminology.

\subsection{Ideals of magmas}

A nonempty subset $I \sub S$ is an ideal of the magma $S$ if $SI$ and $IS$ are contained in~$I$.

\begin{rem}
When an algebra $A$ is strongly graded by a magma with zero~$S$, there is a one-to-one correspondence between homogeneous ideals of~$A$ and ideals of~$S$, given, for ideals $J \normali S$, by $J \mapsto \sum_{j\in J} A_j$.
\end{rem}

\begin{cor}
Let $c$ be a reduced skew set. There is a one-to-one correspondence between the ideals of the skew matrix algebra $\M[n](F;c)$ and the ideals of the frame~$\E(c)$.
\end{cor}

\begin{cor}
A skew matrix algebra $\M[n](F;c)$ has finitely many ideals.
\end{cor}

Let $S$ be a magma. As usual, let $\ideal{X}$ be the minimal ideal containing a subset $X \sub S$. An easy way to find the ideals of a finite magma is to associate a directed graph $\Gamma(S)$, whose vertices are the elements of $S$ with edges $s \ra s'$ if $s' \in \ideal{s}$. Then, a set~$C$ of vertices compose an ideal if and only if it is {\bf{closed}} in the sense that if $s \ra s'$ is an edge and $s \in C$ then $s' \in C$ as well. The magma is simple if and only if the associated graph is {\bf{strongly connected}}, namely there is a directed path from every vertex to every other vertex.

Now consider a skew set $c$, and its associated frame $\E(c)$. Let $\Gamma_c$ be the graph whose vertices are the elements $Fe_{ij}$ of the frame, with edges $Fe_{ij} \ra Fe_{ik}$ and $Fe_{jk} \ra Fe_{ik}$ whenever $c_{ijk} \neq 0$. Although $\Gamma_c$ is a proper subgraph of $\Gamma(\E(c))$, both graphs have the same closed subsets, as the latter is the transitive closure of the former.

\begin{cor}
The skew matrix algebra $\M[n](F;c)$ is simple if and only if $\Gamma_c$ is strongly connected.
\end{cor}

Notice that the graph $\Gamma_c$ can be computed from the frame; it does not care about specific values of the skew set; only about which ones are nonzero. In particular, if all the entries of $c$ are nonzero, then $\E(c) = \E_n$, and the skew matrix algebra has the same ideals as the standard matrix algebras, that is, none.
\begin{cor}\label{cijknon0}
If $c_{ijk} \neq 0$ for every $i,j,k$,  then $\M[n](F;c)$ is simple. Indeed, for any $ij$ and $k\ell$ there is a path $ij \ra jk \ra k\ell$ in $\Gamma_c$.
\end{cor}

Even better,
\begin{cor}\label{cijinon0}
If $c_{iji} \neq 0$ for every $i,j$, then $\M[n](F;c)$ is simple. Indeed, for any $ij$ and $k\ell$ there is a path $ij \ra jj \ra jk \ra kk \ra k\ell$.
\end{cor}

\begin{exmpl}[{Simple $\M[n](F;c)$ such that most $c_{ijk}$ are zero}]\label{simple_with_c=0}
Fix $n$. Let $c$ be the reduced skew set defined by $c_{iji} = 1$ for all $i,j$ and $c_{ijk} = 0$ for any distinct $i,j,k$. Then $\M[n](F;c)$ is simple by \Cref{cijinon0}.
\end{exmpl}

\subsection{Counterexamples}

It is convenient to be able to construct skew sets for a predetermined system of ideals. Indeed, a subset of the frame forms an ideal if and only if it is closed in the graph, so we merely need to dispose of all the edges outgoing from the to-be-closed set.

\begin{rem}[Forcing ideals in skew matrices]\label{algo}
\begin{enumerate}
\item Given subsets~$I_t$ of the set of $1$-dimensional spaces $\set{Fe_{ij} \suchthat 1\leq i,j\leq n}$, construct a skew set $c$ such that each~$I_t$ is an ideal of $\M[n](F;c)$:
For each~$t$, for every~$i,j,k = 1,\dots,n$, if $Fe_{ij} \in I_t$ or $Fe_{jk} \in I_t$, but $Fe_{ik} \not \in I_t$, set $c_{ijk} = 0$. All the other entries of~$c$ are set to~$1$.
\item Listing all principal ideals of a $\M[n](F;c)$: find the closure of each $Fe_{ij}$ in $\Gamma_c$ via standard graph algorithms.
\end{enumerate}
\end{rem}

We can now present several examples for bad behaviour of skew matrix algebras.

\medskip

It is well-known that the square of an ideal in a nonassociative algebra is not always an ideal. Skew matrix algebras are not to immune to this problem.
\begin{exmpl}[An ideal whose square is not an ideal]\label{badsquare}
Set $n = 3$. Consider the skew set $c$ with $c_{iji} = 0$ for $i\neq j$, $c_{231} = c_{312} = 0$, and all the other values being $1$. This is the skew set forced by the assumption that $I  = Fe_{12}+F_{13}+Fe_{23}$ is an ideal of $\M[n](F;c)$; the skew set is produced by the algorithm in \Rref{algo}.

Since $c_{123} \neq 0$, $I^2 = Fe_{13}$, which is not an ideal.
\end{exmpl}

\begin{figure}
$$\xymatrix@C=18pt@R=18pt{
{\circ} \ar@{<->}[r] \ar@{<->}[d] \ar@{<->}@/^2ex/[rr] \ar@{<->}@/_2ex/[dd] & {\circ} \ar@{<->}[r] \ar@{<->}[d] \ar@{<->}@/^2ex/[dd] & {\circ} \ar@{<->}[d] \ar@{<->}@/^2ex/[dd] & &
{\circ} \ar@{->}[r] \ar@{->}[d] \ar@{->}@/^2ex/[rr] \ar@{->}@/_2ex/[dd] & {\circ} \ar@{<->}[r] \ar@{<-}[d] \ar@{<-}@/^2ex/[dd] & {\circ} \ar@{<->}[d] \ar@{<-}@/^2ex/[dd] & &
{\circ} \ar@{<->}[r] \ar@{<->}[d] \ar@{<->}@/^2ex/[rr] \ar@{<->}@/_2ex/[dd] & {\circ} \ar@{<->}[d] \ar@{<->}@/^2ex/[dd] & {\circ} \ar@{<->}@/^2ex/[dd]
                 \\
{\circ} \ar@{<->}[r] \ar@{<->}[d] \ar@{<->}@/^2ex/[rr] & {\circ} \ar@{<->}[r] \ar@{<->}[d] & {\circ} \ar@{<->}[d]& &
{\circ} \ar@{<-}[r] \ar@{->}[d] \ar@{->}@/^2ex/[rr] & {\circ} \ar@{->}[r] \ar@{->}[d] & {\circ} \ar@{<-}[d] & &
{\circ} \ar@{<->}[r] \ar@{<->}@/^2ex/[rr] & {\circ} \ar@{<->}[r] \ar@{<->}[d] & {\circ} \ar@{<->}[d]
\\
{\circ} \ar@{<->}[r] \ar@{<->}@/_2ex/[rr] & {\circ} \ar@{<->}[r] & {\circ} & &
{\circ} \ar@{<-}[r] \ar@{<-}@/_2ex/[rr] & {\circ} \ar@{<-}[r] & {\circ} & &
{\circ} \ar@{<->}@/_2ex/[rr] & {\circ} \ar@{<->}[r] & {\circ}
}$$
\caption{The graphs $\Gamma_c$ for standard matrices (left), for \Eref{badsquare} (middle), and for \Eref{simple_with_c=0} with $n=3$ (right)}
\end{figure}

\def\dontimply{{\ \ \not\!\!\!\implies}}
Let us also comment on skew matrix algebras in the context of the main structural families of finite dimensional nonassociative algebras.

\begin{exmpl}
There are skew matrix algebras which are counterexamples to the following implications:
$$\mbox{skew matrix algebra} \dontimply
\mbox{semiprime} \dontimply
\mbox{prime} \dontimply \mbox{simple}.
$$

\begin{enumerate}
\item The algebra in \Eref{n=2}.\eq{n=2.1} is not semiprime.
\item A semiprime algebra which is not prime: Force the ideals $I = Fe_{13}+Fe_{32}+Fe_{12}$ and $I' = Fe_{24}+Fe_{41}+Fe_{21}$ in a skew matrix algebra of degree $n = 4$ (as in \Rref{algo}). The only principal ideals are $I,I',R,R+Fe_{11}$ and $R+Fe_{22}$, where~$R$ is the space spanned by all matrix units other than $e_{11},e_{22}$.
     The algebra is not prime because $II' = 0$; but the square of any ideal is nonzero.

\item The algebra in \Eref{n=2}.\eq{n=2.2} is prime but not simple.
\end{enumerate}
\end{exmpl}

\section{The nucleus}\label{sec:nucprep}

In this section we collect properties of matrix units in a skew matrix algebra with respect to the nucleus, to be used in the next section for a structural description.

\subsection{The nucleus and one-sided nuclei}\label{ss:nuc}

We have seen in \Pref{dn} that the diagonal subalgebra is always in the nucleus.
\begin{rem}\label{trivi4}
Let $A \sub \M[n](F;c)$ be a subalgebra containing the diagonal subalgebra $\Delta$. Then $A$ is homogeneous, and every ideal of $A$ is homogeneous.

Indeed, for $a \in A$, writing $a = \sum a_{ij}e_{ij}$ ($a_{ij} \in F$), we have that $e_{ii}ae_{jj} = a_{ij}e_{ij} \in A$ for every $i,j$. If $I \normali A$ then $I = \sum e_{ii}I e_{jj} = \sum (I \cap Fe_{ij})$.
\end{rem}

\begin{cor}\label{nucleus-eij}
The nucleus of a skew matrix algebra $\M[n](F;c)$, as well as each of the one-sided nuclei, is homogenous.
\end{cor}
In this sense, each one-sided nucleus can be regarded as a binary relation on the set of indices. By \Pref{dn}, these relations are reflexive.

We prove some symmetry conditions on $c$ which are necessary for matrix units to be in the nucleus. For convenience, we denote the ``reduced associator''
\begin{equation}\label{assocformula-}
(e_{ij},e_{jk},e_{k\ell})_0 = c_{ijk}c_{ik\ell} - c_{ij\ell}c_{jk\ell},
\end{equation}
which is a scalar; the associator is $(e_{ij},e_{jk},e_{k\ell}) = (e_{ij},e_{jk},e_{k\ell})_0e_{i\ell}$.

\begin{prop}\label{73}
Set $A = \M[n](F;c)$. For any $i,j,k$:
\begin{enumerate}
\item\label{73-0} If $\operatorname{span}\set{e_{ii},e_{ij},e_{ji},e_{jj}}$ is associative then $c_{jij} = c_{iji}$.
\item\label{73-1} If $e_{ij}$ is in any of the one-sided nuclei, then $c_{jij} = c_{iji}$.
\item\label{73-2} If $e_{ij} \in \nuc[\ell]{A} \cup \nuc[c]{A}$ then $c_{ijk}c_{jik} = c_{iji}$.
\item\label{73-3} If $e_{ij} \in \nuc[r]{A} \cup \nuc[c]{A}$ then $c_{kij}c_{kji} = c_{iji}$.
\end{enumerate}
\end{prop}
\begin{proof}
By \Eq{assocformula-}, $(e_{ij},e_{ji},e_{ij})_0 = c_{iji}-c_{jij} = - (e_{ji},e_{ij},e_{ji})_0$, which proves the first claim. The second one follows immediately.
Next, $(e_{ij},e_{ji},e_{ik})_0 = c_{iji}-c_{jik}c_{ijk}$ proves the third claim (switching $i,j$ if $e_{ij} \in \nuc[c]{A}$), and
$(e_{ki},e_{ij},e_{ji})_0 = c_{kij} c_{kji} - c_{iji}$ proves the fourth claim.
\end{proof}

\begin{prop}\label{givesymm}
Put $A = \M[n](F;c)$ where $c$ is a skew set.
If $e_{ij} \in \nuc[*]{A}$ for some $* \in \set{\ell,c,r}$, and $c_{iji} \neq 0$, then $e_{ji} \in \nuc[*]{A}$ as well.
\end{prop}
\begin{proof}
First assume $e_{ij} \in \nuc[\ell]{A}$. Namely, $c_{ijk}c_{ik\ell} = c_{ij\ell}c_{jk\ell}$ for any $k,\ell$. Now
$$c_{ijk}(e_{ji},e_{ik},e_{k\ell})_0 = -c_{ji\ell}(e_{ij},e_{jk},e_{k\ell})_0 + c_{jk\ell}(c_{ijk}c_{jik} - c_{ji\ell}c_{ij\ell})$$
is zero by \Pref{73}.\eq{73-2} and the assumption.
Since $c_{iji} \neq 0$, by \Pref{73}.\eq{73-2} we also have that $c_{ijk} \neq 0$, which proves that $e_{ji} \in \nuc[\ell]{A}$.

Next, assume $e_{ij} \in \nuc[r]{A}$. Now
$$c_{\ell ij}(e_{\ell k},e_{kj},e_{ji})_0 = -c_{kji}(c_{\ell k},e_{ki},e_{ij})_0 + c_{\ell kj}(c_{\ell ij}c_{\ell ji} - c_{kij}c_{kji})$$
is zero by \Pref{73}.\eq{73-3} and the assumption, proving again that $e_{ji} \in \nuc[r]{A}$ because
$c_{\ell ij} \neq 0$ by \Pref{73}.\eq{73-3}.

Finally, assume $e_{ij} \in \nuc[c]{A}$. Then
$$c_{ijk}(e_{\ell j},e_{ji},e_{ik})_0 = -c_{\ell ji}(e_{\ell i},e_{ij},e_{jk})_0 + c_{\ell jk}(c_{\ell ij}c_{\ell ji} - c_{ijk}c_{jik})$$
is zero by the assumption and by \Pref{73}.\eq{73-2}--\eq{73-3}, which shows that $e_{ji} \in \nuc[c]{A}$ since $c_{ijk} \neq 0$ by the assumption.
\end{proof}

\begin{cor}
If all $c_{iji} \neq 0$, then the nucleus of $\M[n](F;c)$ is symmetric (as a relation).
\end{cor}

\begin{prop}\label{givetrans}
Put $A = \M[n](F;c)$ where $c$ is a skew set. If $e_{ij}, e_{jk} \in \nuc[*]{A}$ for some $* \in \set{\ell,c,r}$, and $c_{ijk} \neq 0$, then $e_{ik} \in \nuc[*]{A}$.
\end{prop}
\begin{proof}
Each $\nuc[*]{A}$ is a subalgebra so $e_{ij}e_{jk} = c_{ijk}e_{ik} \in \nuc[*]{A}$ by assumption.
\end{proof}

\forget
\begin{enumerate}
\item
The associator:
$$(e_{ij},e_{jk},e_{k\ell}) = (c_{ijk}c_{ik\ell} - c_{ij\ell}c_{jk\ell})e_{i\ell}$$
\item
$(e_{ij},e_{ji},e_{ij}) = (c_{iji}c_{iij} - c_{ijj}c_{jij})e_{ij} = (c_{iji} - c_{jij})e_{ij}$.
This implies that if $e_{ij}$ is in any of the one-sided nuclei, then $c_{iji}=c_{jij}$.
\item
$(e_{ij},e_{ji},e_{ik}) = (c_{iji} - c_{ijk}c_{jik})e_{ik}$.
\end{enumerate}
\forgotten

\subsection{Counterexamples for the one-sided nuclei}
We provide some example of bad behaviour of the nuclei.

\begin{exmpl}[Distinct one-sided nuclei]
Consider the algebra from \Eref{badsquare}. %
Then $\nuc{A} = \sum F e_{ii} + Fe_{13}$, which is not semisimple. Moreover in this case $\nuc[\ell]{A} = \nuc{A} + Fe_{23}+Fe_{32}$,
$\nuc[c]{A} = \nuc{A} + Fe_{13}$,
$\nuc[r]{A} = \nuc{A} + Fe_{12}+Fe_{21}$.
So the one-sided nuclei are distinct, and none is contained in the intersection of the other two.
\end{exmpl}

\begin{exmpl}[Simple algebra with distinct one-sided nuclei]
For $n = 4$, take the trivial skew set $c_{ijk} = 1$, except for the value $c_{123} = 2$. The nucleus is equal to the diagonal, $\Delta$. But $\nuc[\ell]{A} = \Delta+\sum_{i,j \neq 1} Fe_{ij}$;
$\nuc[c]{A} = \Delta+\sum_{i,j \neq 2} Fe_{ij}$; and
$\nuc[r]{A} = \Delta+\sum_{i,j \neq 3} Fe_{ij}$. So the intersections of any two one-sided nuclei are distinct.
\end{exmpl}

\section{Structure of the nucleus}\label{sec:nucofmat}

For any skew matrix algebra, we obtain in this section an explicit combinatorial decomposition as a direct sum of matrix blocks over~$F$ and the radical.

\subsection{Block and partition subalgebras}\label{ss:parti}

Fix a skew set $c$.
\begin{defn}
For every subset $N$ of the index set $\set{1,\dots,n}$, the subalgebra spanned by the $e_{ij}$ with $i,j \in N$ is a (nonunital) subalgebra of $\M[n](F;c)$. These are called {\bf{block subalgebras}}.
\end{defn}
Being a skew matrix algebra of degree $\card{N}$, we may apply various results on a block subalgebra to obtain information on the ambient skew set. For example, if the block subalgebra of $\set{i,j}$ is associative, then $c_{iji}=c_{jij}$ (this is \Pref{73}.\eq{73-0}). If the block subalgebra of $\set{i,j,k}$ is
associative and simple, then necessarily $c_{ijk} \neq 0$ (\Pref{nonzeroassoc}).

Similarly,
\begin{defn}
For an equivalence relation $\equiv$ on the index set, the subalgebra spanned by the $e_{ij}$ with $i \equiv j$ is a (unital) subalgebra of $\M[n](F;c)$, which is the direct sum of the block subalgebras defined by the equivalence classes of the relation. Such subalgebras are called {\bf{partition subalgebras}}.
\end{defn}

\begin{rem}\label{whatissemisimple}
Being a direct sum of its block subalgebras, a partition subalgebra is semisimple if and only if all of its blocks are simple.
\end{rem}

Generically, the nucleus of a skew matrix algebra is the diagonal, which is the smallest partition subalgebra. Our goal is to characterize the skew matrix algebras whose nucleus is a semisimple partition subalgebra.

\subsection{Semisimple nucleus}

Fix a skew set $c$. A homogeneous subalgebra $A \sub \M[n](F;c)$ is {\bf{regular}} if $c_{iji} \neq 0$ for every $e_{ij} \in A$. Regularity is inherited by subalgebras. If $\M[n](F;c)$ itself is regular then it is simple by \Cref{cijinon0}; but there can be proper subalgebras which are regular even when $\M[n](F;c)$ is not regular. The diagonal subalgebra is always regular.

\begin{rem}
Recall that an associative algebra $A$ is von Neumann regular if every element $a \in A$ satisfies $a \in aAa$. Our definition of regularity is motivated by the observation that  $e_{ij} \in (e_{ij}\M[n](F;c))e_{ij}$ if and only if $c_{iji} \neq 0$. Since we are mostly concerned with the nucleus, for whose elements \Pref{73}.\eq{73-1} holds, we will obtain the same statements if regularity was taken to mean $c_{jij} \neq 0$ for the $e_{ij} \in A$.
\end{rem}

We will prove the following.
\begin{thm}\label{main7}
Let $c$ be a skew set. The nucleus of $\M[n](F;c)$ is a semisimple partition subalgebra if and only if the nucleus is regular.
\end{thm}
Note that for the nucleus, being associative, ``semisimple'' means sum of matrix blocks.
The statement of \Tref{main7} is true (with the same proof) for any of the one-sided nuclei, as well as for an intersection of any two one-sided nuclei.
\begin{cor}
Assume $c_{iji} \neq 0$ (all $i,j$). Then~$\M[n](F;c)$ is simple and has a semisimple nucleus.
\end{cor}
Indeed, in this case $\M[n](F;c)$ is regular, so the nucleus is regular as well. We first prove the following.

\begin{prop}\label{quitenice}
A regular partition subalgebra $P$ of $\M[n](F;c)$ is semisimple.
\end{prop}
\begin{proof}
If $i,j$ are in the same block, then $e_{ij} \in P$ so $c_{iji} \neq 0$ by assumption and then the corresponding block subalgebra is simple by \Cref{cijinon0}. The partition algebra, being a direct sum of its block subalgebras, is thus semisimple.
\end{proof}

\begin{prop}\label{main7.1}
If the nucleus of $\M[n](F;c)$ is regular, then it is a semisimple partition subalgebra.
\end{prop}
\begin{proof}
Denote $M = \M[n](F;c)$. Assume $\nuc{M}$ is regular. If $e_{ij} \in \nuc{M}$, then $c_{iji} \neq 0$ by assumption, and $e_{ji} \in \nuc{M}$ by \Pref{givesymm}. Now suppose $e_{ij}, e_{jk} \in \nuc{M}$; by \Pref{73}.\eq{73-2} (applied to $e_{ij}$) or \Pref{73}.\eq{73-3} (applied to $e_{jk}$), $c_{ijk} \neq 0$ and so by \Pref{givetrans} we obtain $e_{ik} \in \nuc{M}$. This shows that the relation defining $\nuc{M}$ is symmetric and transitive, so $\nuc{M}$ is a partition subalgebra.

Now the nucleus is a regular partition subalgebra, so it is semisimple by \Pref{quitenice}.
\end{proof}

\begin{prop}\label{main7.2}
Let $A \sub \M[n](F;c)$ be a homogeneous associative subalgebra containing the diagonal.
If $A$ is semisimple, then it is regular.
\end{prop}
\begin{proof}
Since $A$ contains the diagonal, any ideal $I\vartriangleleft A$ is homogeneous by
\Rref{trivi4}.

Suppose $e_{ij} \in A$. Let $I = A e_{ij} A$, a nonzero ideal.
Now $0 \neq I^2 = Ae_{ij}Ae_{ij}A$ because $A$ is semiprime, so necessarily $e_{ji} \in A$ and $e_{ij}Ae_{ij} = F e_{ij}e_{ji}e_{ij}$ is nonzero. In particular $c_{iji} \neq 0$, as needed.
\end{proof}

We can now prove the main theorem of this section:

\begin{proof}[Proof of \Tref{main7}]
The direction $(\Leftarrow)$ is \Pref{main7.1}. For the direction $(\Rightarrow)$, note that $\nuc{\M[n](F;c)}$ is a homogeneous associative subalgebra containing the diagonal, so we are done by \Pref{main7.2}.
\end{proof}

\subsection{The nucleus of skew matrices}

Fix a skew set $c$. In the previous subsection we gave conditions for the nucleus $N = \nuc{\M[n](F;c)}$ to be a partition semisimple subalgebra.
Let $J = \Jac(N)$ be the radical of~$N$. Here we describe $J$ and realize the semisimple quotient $N/J$ as a partition subalgebra.

\begin{lem}\label{whoisnilp}
A matrix unit $e_{ij} \in N$ generates a nilpotent ideal of~$N$ if and only if $c_{iji} = 0$.
\end{lem}
\begin{proof}
If $c_{iji} = 0$ then $(Ne_{ij}N)^2 = Ne_{ij}Ne_{ij}N \sub N e_{ij}e_{ji}e_{ij}N = 0$ as in \Pref{main7.2}. On the other hand, assume $c_{iji} \neq 0$. Then $e_{ji} \in N$ by \Pref{givesymm}, and $c_{iji}e_{ii} = e_{ij}e_{ji} \in e_{ij}N$, so $Ne_{ij}N$ contains the idempotent $e_{ii}$ and cannot be nilpotent.
\end{proof}

Let $J = \Jac(N)$ denote the radical of~$N$.
\begin{prop}\label{firstchar}
The radical $J$ is spanned by the $e_{ij} \in N$ for which $c_{iji} = 0$.
\end{prop}
\begin{proof}
By \Rref{trivi4}, $J$ is homogeneous, so it is spanned by matrix units. But $J$ is the largest nilpotent ideal, so by \Lref{whoisnilp} it is spanned by the $e_{ij}$ with $c_{iji} = 0$.
\end{proof}
It follows that the nucleus is regular if and only if $J = 0$, which is the core of \Tref{main7}.

Let $S$ be the subspace of $N$ spanned by the matrix units $e_{ij} \in N$ for which $c_{iji} \neq 0$ (recall that when $e_{ij} \in N$ we have that
$c_{iji} = c_{jij}$). The definitions in terms of the skew set provide an obvious decomposition of vector spaces, $N = S \oplus J$.
Note that $\Delta \sub S$.
We will see that $S$ is the maximal regular homogeneous subalgebra of the nucleus.

\begin{prop}\label{aboutS}
The space $S$ is a maximal semisimple subalgebra of the nucleus.%
\end{prop}
\begin{proof}
To prove that $S$ is a subalgebra of $N$, let $e_{ij},e_{jk}\in S$, i.e.\ $e_{ij},e_{jk}\in N$ and $c_{iji},c_{jkj}\neq 0$. By \Pref{73}.\eqref{73-2}, $c_{ijk}\neq 0$, and thus we need to show that $e_{ik}\in S$. We have that $e_{ik}\in N$ by \Pref{givetrans}, so we are left with proving that $c_{iki}\neq 0$. But since $e_{ij}\in N$ we have that $(e_{ij},e_{jk},e_{ki})_0=c_{ijk}c_{iki}-c_{iji}c_{jki}=0$, and since $c_{jki}c_{kji} = c_{jkj} \neq 0$ (using $e_{jk} \in N$ in \Pref{73}.\eqref{73-2}) we obtain that $c_{ijk}c_{iki} = c_{iji}c_{jki} \neq 0$, from which the claim follows.

This computation shows that the relation defining $S$ is transitive; the relation is symmetric because $c_{jij} = c_{iji}$ whenever $e_{ij} \in N$. Together, this shows that $S$ is a partition algebra, which, being clearly regular, is semisimple by \Pref{quitenice}.

Now, if $S \sub S' \sub N$ is an intermediate semisimple algebra, then $S' \cap J$ is a nilpotent ideal of $S'$, so $S' \cap J = 0$,
but $S'$ is homogeneous using \Rref{trivi4} because $\Delta \sub S \sub S'$, and so $S' = S$ because $N = S \oplus J$.
\end{proof}

We obtain a homogeneous realization of Wedderburn's principal \linebreak theorem:
\begin{thm}\label{sumdegjac}
For any skew set $c$, $$\nuc{\M[n](F;c)} = S \oplus \Jac(\nuc{\M[n](F;c)}),$$
where $S$ is a partition subalgebra, namely a direct sum of matrix algebras over $F$, whose sum of degrees is~$n$.
\end{thm}

\begin{exmpl}
Although~$S$ is clearly the unique maximal {\emph{homogenous}} semisimple subalgebra of $\nuc{\M[n](F;c)}$, the choice of a semisimple complement for~$J$ is not unique.

For example, consider the skew matrix algebra $\M[2](F;c)$ with $c_{121} = c_{212} = 0$, which is associative (see \Eref{n=2}.\eq{n=2.1}). The algebra $S = Fe_{11}+Fe_{22}$ is isomorphic to $F(e_{11}+p) + F(e_{22}-p)$ for any $p \in J = Fe_{12}+Fe_{21}$.
\end{exmpl}

\subsection{Counterexamples for the nucleus}

Once more, we conclude this section with several example of bad behavior, this time of the nucleus.

\begin{exmpl}[The nucleus in degree $n = 2$]
For any nonassociative skew matrix algebra of degree $2$, the nucleus is the diagonal subalgebra.

Indeed, if $e_{12}$ is in the nucleus then $c_{121}=c_{212}$ by \Pref{73}.\eq{73-1}; but then the algebra is associative.
\end{exmpl}

Specialization of this example demonstrates a delicate point in \Tref{main7}, namely, that the nucleus may be a partition subalgebra
but not semisimple.
\begin{exmpl}[A partition subalgebra nucleus which is not semisimple]\label{Ex5}
The skew matrix algebra of degree $2$ with $c_{121}=c_{212} = 0$ is associative, hence the nucleus is the whole algebra, but is not semisimple (see \Rref{n=2}.\eq{n=2.2}).
\end{exmpl}

Next, we show that the structural properties of being simple and having a semisimple nucleus are independent.  The algebra in \Eref{Ex5} shows that
\begin{center}
semisimple nucleus $\ \ \not\!\! \Longrightarrow \ $  simple.
\end{center}
On the contrary, we have:
\begin{exmpl}[simple $\ \ \not\!\! \Longrightarrow\ $  semisimple nucleus]\label{nonnormal}
Consider $\M[4](F;c)$, where $c_{ijk}$ are defined by
\begin{align*}
&c_{121}=c_{123}=c_{141}=c_{143}=c_{212}=c_{214}=c_{232}=c_{234}=\\
&=c_{321}=c_{323}=c_{341}=c_{412}=c_{414}=c_{423}=c_{432}=0
\end{align*}
and the other $c_{ijk}=1$. A straightforward verification (done by an ad-hoc computer program) of the graph $\Gamma_c$ shows that $\M[4](F;c)$ is simple. However, the nucleus is $\Delta+Fe_{12}$ where $\Delta$ is the diagonal subalgebra.
\end{exmpl}
In this example the nucleus is not a partition subalgebra.

\section{Forms of skew matrices}\label{sec:forms}

We can now show that the nucleus of any semiassociative algebra decomposes. This is used to show that semiassociative algebras are precisely the forms of skew matrices.

Let us first verify that semiassociativity is stable under scalar extension.

\begin{prop}\label{basicT0}
A scalar extension of a semiassociative algebra is semiassociative.
\end{prop}
\begin{proof}
Let $A$ be a semiassociative $F$-algebra of dimension $n^2$, with an $n$-dimensional \etale{} subalgebra $K$ contained in the nucleus, such that~$A$ is $K^e$-faithful where $K^e = K \tensor K$. Let~$E$ be a field extension of~$F$. The nonassociative algebra $A_E = E \tensor[F] A$ has dimension $n^2$ over $E$, and is central; its nucleus $\nuc{E \tensor A} = E \tensor \nuc{A}$ contains $K_E = E \tensor K$ which is \etale{} and has dimension~$n$ over~$E$; and $A_E$ is cyclic over $(K_E)^e = (E \tensor[F] K) \tensor[E] (E \tensor[F] K) = E \tensor[F] (K \tensor[F] K) = (K^e)_E$ because~$A$ is cyclic over $K^e$.%
\end{proof}

To every nonassociative algebra $A$ there is a naturally associated semisimple algebra, namely the quotient
\begin{equation}\label{sigmadef}
\s(A) = \nuc{A}/\Jac(\nuc{A}).
\end{equation}
This quotient will play a key role in the sequel.

\begin{prop}\label{samesigma}
Let $A$ be any nonassociative algebra. For any separable field extension $E/F$ we have that
$$\s(E \tensor A) = E \tensor \s(A).$$
\end{prop}
\begin{proof}
We have that $\nuc{E \tensor A} = E \tensor \nuc{A}$, and since $E/F$ is separable, $\Jac(E \tensor \nuc{A}) = E \tensor \Jac(\nuc{A})$
\cite[Theorem~(5.17)]{Semisimple}.  %

It follows that
$$\s(E \tensor A) = (E \tensor \nuc{A})/(E \tensor \Jac(\nuc{A})) = E \tensor \nuc{A}/\Jac(\nuc{A}) = E \tensor \s(A).$$
\end{proof}

Recall that an associative $F$-algebra $R$ is {\bf{separable}} if $R$ is projective as a module over $R^e = R \tensor \op{R}$. A finite dimensional algebra is separable if and only if it is semisimple and its center is \etale{} over $F$.
\begin{thm}\label{s(A)sep}
The semisimple quotient $\s(A)$ of the nucleus of a semiassociative algebra~$A$ is separable.
\end{thm}
\begin{proof}
Let $A$ be a semiassociative algebra of degree~$n$, with an \etale{} $n$-dimensioanl subalgebra $K \sub \nuc{A}$. Let $E$ be a separable splitting field of $K$. By \Tref{main-thm}, $E \tensor A$ is a skew matrix algebra.
In \Cref{sumdegjac} we proved that for a skew matrix algebra $\M[n](F;c)$, the semisimple quotient $\s(\M[n](F;c))$ is split, namely a direct sum of matrix algebras over $F$, and in particular separable.
But an algebra which becomes separable after a field extension was separable to begin with~\cite{DI},   %
so by \Pref{samesigma} we proved that $\s(A)$ is separable.
\end{proof}

\begin{cor}\label{WeddNucl}
For any semiassociative algebra $A$ we have a decomposition
$$\nuc{A} = \s(A) \oplus \Jac(\nuc{A}).$$
In other words, $\s(A)$ embeds in $A$ as a (nuclear) subalgebra.
\end{cor}
This is Wedderburn's principal theorem, which applies to any finite dimensional algebra whose semisimple quotient is separable.

We are now concerned with \etale{} subalgebras of the nucleus. Let $\delta(T)$ denote the maximal dimension of an \etale{} $F$-subalgebra of a semisimple $F$-algebra $T$. By \cite[Theorem~2.5.9]{Jac} we have that $\delta(\M[n](F)) = n$.

\begin{prop}\label{shouldbe=}
Let $T$ be a separable $F$-algebra, and $E$ a separable splitting field. Then $\delta(T) = \delta(E \tensor T)$.
\end{prop}
\begin{proof}
An \etale{} subalgebra of $T$ extends to an \etale{} subalgebra of $E \tensor T$, so clearly $\delta(T) \leq \delta(E \tensor T)$.

Let us prove the other direction. Since $\delta$ is additive, we may assume~$T$ is simple, so that $T = \M[t](D)$ for a division algebra $D$ whose center,~$L$, is finite dimensional and separable over~$F$. Since~$E$ splits~$L$ and~$D$, we have that $E \tensor T =  \M[td](E)^{\oplus [L:F]}$ where $d = \deg(D)$. Then $\delta(E \tensor T) = \dimcol{L}{F} \cdot \delta(\M[td](E)) = td \dimcol{L}{F}$. Let $K_0$ be a maximal separable $L$-subfield of $D$; then~$K_0 \times \cdots \times K_0$, with~$t$ copies embedded diagonally in $T$, has dimension~$td$ over~$L$ and thus dimension $\delta(E \tensor T)$ over~$F$.
\end{proof}

\begin{thm}
Let $A$ be any nonassociative $F$-algebra such that $E \tensor A$ is semiassociative for a separable extension $E/F$. Then $A$ is semiassociative.
\end{thm}
\begin{proof}
Let $n$ denote the degree of $E \tensor A$. By \Pref{samesigma}, $E \tensor \s(A) = \s(E \tensor A)$. This algebra is separable by \Tref{s(A)sep}, so it follows that $\s(A)$ is separable as well. Let $K'$ be an \etale{} $E$-subalgebra  of dimension~$n$ in the nucleus of~$E \tensor A$. Let $E'/E$ be a splitting field of $K'$. Then $E' \tensor[E] (E \tensor[F] A) = E' \tensor[F] A$ is a skew matrix algebra, and by \Cref{sumdegjac} we have that $$\delta(\s(E' \tensor A)) = \deg(E' \tensor A) = n.$$

Now \Pref{shouldbe=} gives us that $$\delta(\s(A)) = \delta(E' \tensor \s(A)) = \delta(\s(E' \tensor A)) = n,$$ so $\s(A)$ contains an \etale{} $n$-dimensional $F$-subalgebra~$K$. This subalgebra is contained in $\nuc{A}$ by separability of~$\s(A)$. Now $A$ must be faithful over $K^e$, since $E \tensor A$ is faithful over $(E \tensor K)^e$ by \Pref{coverall}.
\end{proof}

Calling an $F$-algebra $A$ a {\bf{form}} of $A'$ if $E \tensor A = A'$ for some separable field extension $E/F$, and building on \Tref{main-thm}, we thus proved:
\begin{cor}
An algebra is semiassociative if and only if it is a form of skew matrices.
\end{cor}

\part{Tensor products of semiassociative algebras}\label{Part3}

The class of semiassociative algebras is closed under taking tensor products (\Tref{basicTx} below). By \Rref{tensormat}, this is also the case for the subclass of skew matrix algebras. Considering several intermediate classes characterized by properties of the nucleus, we define in this part a semiassociative Brauer monoid, which contains the classical Brauer group as a unique maximal subgroup.

\section{Closure under tensor products}\label{sec:tensor}

We have the following fundamental observation.
\begin{thm}\label{basicTx}
The class of semiassociative algebras is closed under tensor products over a common center.
\end{thm}
\begin{proof}
The key here is the fact that tensor product preserves faithfulness of modules, see \cite[Thm.~1]{Bergman1} or \cite[Lem.~1.1]{Passman1}.

Let $A$ and $A'$ be semiassociative $F$-algebras, of dimensions~$n^2$ and~$n'^2$; with \etale{} subalgebras $K \sub \nuc{A}$ and $K' \sub \nuc{A'}$ of dimensions~$n$ and~$n'$; and such that $A$ and $A'$ are faithful over $K^e$ and $K'^e$, respectively. The algebra $A \tensor[F] A'$, of dimension $(nn')^2$, is $F$-central; its nucleus $\nuc{A \tensor A'} = \nuc{A} \tensor \nuc{A'}$ contains the $nn'$-dimensional \etale{} algebra $K \tensor K'$, and is faithful over $(K \tensor K')^e = (K \tensor K') \tensor (K \tensor K') = (K \tensor K) \tensor (K' \tensor K') = K^e \tensor K'^e$.
\end{proof}

\begin{cor}\label{matricesup}
If $A$ is semiassociative over $F$, then $\M[m](A)$ is semiassociative over $F$ as well.
\end{cor}

As we saw in \Pref{prop:nuc-tensor}, taking the nucleus commutes with tensor products, i.e.\ $\nuc{A\otimes B}=\nuc{A}\otimes\nuc{B}$. We show here that in the case of semiassociative algebras, this is not only true for the nuclei, but also for their semisimple quotients.

\begin{lem}\label{centers}
Let $S_1,S_2$ be finite dimensional  associative simple $F$-algebras with centers $K_i = \Zent(S_i)$. If $K_1 \tensor[F] K_2$ is semisimple then $S_1 \tensor S_2$ is semisimple as well.
\end{lem}
Recall that the tensor product of two fields is semisimple unless both contain a common nonseparable subfield.
\begin{proof}
If $K_1 = F$ then we are done by \cite[Prop.~24.9]{RowenNCV} (indeed the tensor product of $S_1 \tensor[F] S_2$ with $\op{S_1}$ over~$F$ is a matrix algebra over~$S_2$ which is simple).

For the general case, decompose $K_1 \tensor K_2 = L_1 \oplus \cdots \oplus L_t$ where each~$L_j$ is a field containing both~$K_1$ and~$K_2$. Now $$S_1 \tensor[F] S_2 = (S_1 \tensor[K_1] K_1) \tensor[F] (K_2 \tensor[K_2] S_2) = S_1 \tensor[K_1] ((K_1 \tensor[F] K_2) \tensor[K_2] S_2)$$
 is a direct sum of the $S_1 \tensor[K_1] (L_j \tensor[K_2] S_2)$, so we reduced to the previous case.
\end{proof}

\begin{cor}\label{95}
The tensor product of separable algebras over $F$ is semisimple.
\end{cor}

\begin{prop}\label{s(A)-tens}
Let $A,B$ be two $F$-semiassociative algebras. Then
$$\Jac(\nuc{A\otimes B})=\Jac(\nuc{A})\otimes\nuc{B}+\nuc{A}\otimes\Jac(\nuc{B}),$$
and thus
$$\s(A\otimes B)\cong\s(A)\otimes\s(B).$$
\end{prop}
\begin{proof}
By \Tref{s(A)sep}, the semisimple quotients $\s(A)$ and $\s(B)$ are separable. Hence by Wedderburn's principal theorem, we may write $\nuc{A}\cong \s(A)\oplus \Jac(\nuc{A})$ and $\nuc{B}\cong \s(B)\oplus\Jac(\nuc{B})$. Also, by \Cref{95}, $\s(A)\otimes\s(B)$ is semisimple. Therefore
$$\nuc{A\otimes B}\cong(\s(A)\otimes \s(B))\oplus\left(\Jac(\nuc{A})\otimes\nuc{B}+\nuc{A}\otimes\Jac(\nuc{B})\right)$$
is a vector space decomposition of $\nuc{A\otimes B}$ to a semisimple algebra and a nilpotent ideal, showing that the second direct summand is precisely $\Jac(\nuc{A\otimes B})$ and the first one is isomorphic to $\s(A\otimes B)$.
\end{proof}

\section{The semiassociative Brauer Monoid}\label{section:brauer monoid}

We now define the ``semiassociative Brauer monoid''. Recall that the elements of the classical Brauer group of a field~$F$ are the equivalence classes of associative central simple algebras over $F$, where algebras $A,A'$ are Brauer equivalent if $A\otimes\M[n'](F)\isom A'\otimes\M[n](F)$ for some $n,n'$. %
\begin{defn}\label{ourdef}
We say that semiassociative algebras $A$ and~$B$ over~$F$ are {\bf{Brauer equivalent}}, denoted $A\sim B$, if there are skew matrix algebras $\M[n](F;c)$ and $\M[n'](F;c')$ such that
$$A\tensor[F]\M[n](F;c)\cong B\tensor[F]\M[n'](F;c').$$
\end{defn}
Using \Rref{tensormat} this is readily seen to be an equivalence relation.
We denote the equivalence class of a semiassociative algebra~$A$ by~$[A]^{\sa}$.\smallskip

\begin{rem}\label{ImOK}
Let $A,A',B,B'$ be semiassociative algebras, such that $A \sim A'$ and $B \sim B'$. Then $A \tensor A' \sim B \tensor B'$.

This follows from the tensor product of skew matrix algebras being a skew matrix algebra.
\end{rem}

\begin{defn}
The {\bf{semiassociative Brauer monoid}} of a field~$F$, denoted $\Br[sa](F)$, is the set of equivalence classes of semiassociative $F$-algebras under equivalence, with the product $\nacl{A} \cdot \nacl{B} = \nacl{A \tensor B}$.
\end{defn}
The operation is well-defined by \Rref{ImOK}. We thus defined a monoid, whose identity element is $\nacl{F}$.

\smallskip
For associative algebras, this new definition of Brauer equivalence coincides with the classical one:
\begin{prop}\label{goodcompare}
Let $A,B$ be associative central simple algebras over a field $F$, which are equivalent according to \Dref{ourdef}. Then~$A$ and~$B$ are Brauer equivalent in the classical sense.
\end{prop}
\begin{proof}
Suppose $A \tensor \M[n](F;c) \isom B \tensor \M[n'](F;c')$ for skew matrix algebras $\M[n](F;c)$ and $\M[n](F;c')$. Taking the semisimple quotient of the nucleus and using \Pref{s(A)-tens}, we obtain
\begin{equation}\label{this=}
A\tensor \s(\M[m](F;c)) \isom B\tensor \s(\M[n](F;c')).
\end{equation}

By \Cref{sumdegjac}, the left- and right-hand sides in \eq{this=} are direct sums of matrix algebras over $A$ and~$B$, respectively. The decomposition into direct sum of a semisimple associative algebra is unique, so comparing a %
component on the left-hand side with the respective isomorphic component at the right-hand side we conclude that $A$ and~$B$ are Brauer equivalent in the classical sense.
\end{proof}

By \Pref{goodcompare}, there is a well defined embedding of monoids
\begin{equation}\label{BrBr}
\Br(F) \hookrightarrow \Br[sa](F)
\end{equation}
where $\Br(F)$ is the classical Brauer group of the field, sending the Brauer class $[A]$ of an associative central simple algebra $A$ to $\nacl{A}$. The image of this map consists of the classes in $\Br[sa](F)$ which have at least one associative member.

\begin{exmpl}
The semiassociative Brauer monoid of an algebraically closed field is trivial. For example $\Br[sa](\C) = 1$.

Indeed, the only \etale{} algebras over~$F$ are of the split ones, of the form~$F^n$, and by \Cref{split-nuc-skew} any semiassociative $F$-algebra splits.
\end{exmpl}

\begin{rem}[on quotients of monoids]
\label{ss:mon}
Similarly to the classical Brauer group, the nonassociative Brauer group is defined as a monoid of algebras modulo a congruence relation. Let us make this explicit. Let~$C$ be a commutative semigroup. A semigroup~$C_0$ induces a congruence relation on~$C$, defined by setting $x \sim y$ if $\alpha x = \beta y$ for some $\alpha, \beta \in C_0$. The quotient space, which we denote by $C/C_0$, is a semigroup, whose identity element may strictly contain $C_0$. (Semigroup theorists often use another congruence relation, contracting $C_0$ to a point and leaving the elements of $C-C_0$ distinct.)
\end{rem}
Now, letting $\MM{semi}$ be the monoid of all semiassociative $F$-algebras and $\MM{mat}$ the monoid of skew matrices, $\Br[sa](F)$ is by definition the quotient $\MM{semi}/\MM{mat}$.

\section{Interlude: Semisimple associative algebras}\label{sec:ss}

In this section we consider associative algebras, anticipating their role as the nuclei of semiassociative algebras.
Let~$\SS{sep}$ be the family of separable algebras over $F$, which is a monoid under the tensor product by
\Cref{95}. We say that a (finite dimensional) separable algebra is {\bf{semicentral}} if it is a direct sum of $F$-central simple algebras. Let~$\SS{semicentral}$ be the family of such algebras, which is again a monoid because the tensor product of $F$-central simple algebras is $F$-central.

We will identify invertible elements in the Brauer monoid using the notion of purity. If $C$ is a semigroup and $C_0$ a subsemigroup, we say that $C_0$ is {\bf{pure}} if the complement $C-C_0$ is an ideal; namely if $ab \in C_0$ for $a,b \in C$ implies $a,b \in C_0$.

\begin{prop}\label{pure1}
$\SS{semicentral}$ is a pure submonoid of $\SS{sep}$.
\end{prop}
\begin{proof}
Let $S$ and $S'$ be simple components of the semisimple algebras~$N$ and~$N'$, respectively. The product $S \tensor S'$, which is a direct summand of $N \tensor N'$, decomposes into simple direct summands, whose centers each contain both $\Zent(S)$ and $\Zent(S')$. Since we assume $N \tensor N' \in \SS{semicentral}$, we must have that $\Zent(S) = \Zent(S') = F$, so that $N, N' \in \SS{semicentral}$.
\end{proof}

For a finer analysis, let us define the {\bf{weight}} $\w(S)$ of an algebra $S \in \SS{semicentral}$ to be the number of its central simple components, counted as elements in the Brauer group $\Br(F)$.

\begin{prop}\label{onw-}
Let $N,N' \in \SS{semicentral}$ be semisimple algebras with central components. Then $\w(N \tensor N') \geq \w(N)$.
\end{prop}
\begin{proof}
If $\set{\alpha_1,\dots,\alpha_t}$ are the Brauer classes of components of $N$, and~$\beta$ is any Brauer class of a component of $N'$,
then, as in \Cref{supernormalx}, $\set{\alpha_1 \beta, \dots, \alpha_t \beta}$ are Brauer classes of components of $N \tensor N'$.
\end{proof}

\forget %
\begin{rem}\label{rem:Haile}
According to Haile~\cite{Haile2} (who restricted the definition to simple algebras), a separable algebra~$B$ over $F$ is {\bf{normal}} if $B \tensor[F] \op{B}$ is a direct sum of matrix algebras (over respective centers). The simple normal algebras are the cornerstone of his theory,
as their equivalence classes form Haile's Brauer monoid. A semicentral algebra is normal if and only if it has weight~$1$, because for $B = \oplus B_i$ the simple components of $B \tensor \op{B}$ are the products $B_i \tensor \op{B_{i'}}$, so weight~$1$ requires all blocks to be similar.
\end{rem}
\forgotten

We say that a semicentral algebra~$N$ is {\bf{homogeneous}} if its components are similar to each other in the Brauer sense, namely if $\omega(N) = 1$. (This has nothing to do with homogeneity with respect to the grading of subspaces of skew matrices.)
Let $\SS{hom}$ be the class of semicentral homogeneous algebras.
This is a submonoid, because the components of $N \tensor N'$ are tensor products of components of~$N$ and components of~$N'$.
It follows from \Pref{onw-} that:
\begin{cor}\label{pure2}
The submonoid $\SS{hom}$ is pure in $\SS{semicentral}$.
\end{cor}

\begin{exer}\label{puretransitive}
Purity is transitive: for semigroups $C_0 \sub C_1 \sub C$, if~$C_0$ is pure in~$C_1$ and~$C_1$ is pure in~$C$, then~$C_0$ is pure in~$C$.
\end{exer}

Finally, let $\SS{split}$ be the split semisimple algebras, namely direct sums of matrix algebras over $F$, which is again a monoid. We formed the chain
\begin{equation}\label{SS}
\SS{split} \sub \SS{hom} \sub \SS{semicentral} \sub \SS{sep}.
\end{equation}

\begin{rem}
The classical Brauer group can be viewed as the quotient $$\Br(F) \isom \SS{hom}/\SS{split}.$$
\end{rem}

\section{Semicentral semiassociative algebras}\label{sec:semicentral}

We now return to semiassociative $F$-algebras.
For a semiassociative algebra $A$, the simple components of the semisimple quotient $\s(A)$ will be called the \textbf{atoms} of~$A$. The atoms are simple associative $F$-algebras. The centers of the atoms of $A$ are separable extensions of $F$ by \Tref{s(A)sep}. In this section we focus on algebras with central atoms.

\begin{defn}
A semiassociative algebra over $F$ is called {\bf{semicentral}} if all of its atoms are $F$-central.
\end{defn}
Namely, a semiassociative algebra $A$ is semicentral if $\s(A)$ is semicentral in the sense of \Sref{sec:ss}.

So the atoms of a semicentral semiassociative algebra are $F$-central simple associative algebras. Equivalently, a semiassociative algebra $A$ is semicentral if and only if the center $\Zent(\s(A))$ splits (although its dimension would usually be smaller than $\deg(A)$).

An associative central simple algebra over $F$ is always semicentral, as the algebra itself is its own single atom. Also, by \Cref{sumdegjac}, any skew matrix algebra over $F$ is semicentral.
A non-split nonassociative quaternion algebra is not semicentral, because the nucleus is a quadratic extension of~$F$. The associative zero quaternion algebra in \Eref{Q=0} is not semicentral unless $K$ splits.

We first show that this class of semiassociative algebras is closed under basic manipulations.

\begin{prop}
A scalar extension of a semicentral semiassociative algebra along a separable field extension is semicentral.
\end{prop}
\begin{proof}
When $E$ is separable, $\Zent(\s(E \tensor A)) = E \tensor \Zent(\s(A))$ by \Pref{samesigma}, and $\Zent(\s(A))$ is split.
\end{proof}

\begin{prop}\label{supernormalx}
The tensor product of two semicentral semiassociative algebras is semicentral.
\end{prop}
\begin{proof}
Suppose $A$ and~$B$ are semicentral semiassociative algebras. By \Pref{s(A)-tens} the atoms of $A \tensor B$ are the tensor products of atoms of~$A$ and~$B$, which are central simple by assumption.
\end{proof}

Extending \Pref{sumdegjac} on the sum of degrees of the atoms in a skew matrix algebra, we have the following observation regarding the degrees of the atoms:

\begin{lem}\label{sumup}
The sum of the degrees of the atoms of a semicentral semiassociative algebra~$A$ is equal to the degree of~$A$.
\end{lem}
\begin{proof}
Let $\s(A) = S_1 \oplus \cdots \oplus S_t$ be the semisimple quotient of its nucleus. Let $E$ be a splitting field of~$A$. Then $E \tensor A$ is a semicentral skew matrix algebra, and its atoms, which are matrix algebras over~$E$
with sum of degrees equal to $\deg(A)$, have the same degrees as the~$S_j$. We are done by \Pref{sumdegjac}.
\end{proof}

We can now use atoms to characterize skew matrices.
\begin{cor}\label{use52}
A semiassociative algebra is a skew matrix algebra if and only if its atoms are all central matrix algebras.
\end{cor}
\begin{proof}
The atoms of a skew matrix algebra are matrix algebras by \Pref{sumdegjac}. On the other hand, if the semisimple quotient of the nucleus of a semiassociative algebra~$A$, is the direct sum $\M[n_1](F)\oplus \dots \oplus \M[n_t](F)$, then, since $\sum n_i = \deg(A)$ by \Lref{sumup} and $\s(A)$ is contained in $\nuc{A}$ by \Cref{WeddNucl}, the nucleus contains $F^{\deg(A)}$, so the algebra splits by \Cref{split-nuc-skew}.
\end{proof}

\section{Monoids of semiassociative algebras}\label{sec:monoids}

The monoids~$\MM{semi}$ of all semiassociative algebras and~$\MM{mat}$ of skew matrices were defined after \Rref{ss:mon}.

\forget
For context, we add the following:
\begin{exer}
Let $P \sub M$ be commutative monoids. The relation $x \sim y$ (for $x,y \in M$) when $x p = y q$ for some $p,q \in P$ is an equivalence relation, and the quotient space $M/P$ is a monoid.
\end{exer}
\forgotten

Let the {\bf{radical envelope}} of a partition of $n$ over $F$ be the skew matrix algebra $\M[n](F;c)$ where $c_{ijk} = 1$ for $i,j,k$ in the same block of the partition, or when $j \in \set{i,k}$, and $c_{ijk} = 0$ otherwise. This is easily seen to be an associative algebra, whose semisimple quotient is a partition subalgebra (corresponding to the same partition).

\begin{prop}\label{Guypromisedaproof}
The map
$${\s} \co \MM{semi} \ra \SS{sep},$$
taking an algebra to the semisimple quotient of its nucleus, is a surjective homomorphism.
\end{prop}
\begin{proof}
The map is a homomorphism by \Pref{s(A)-tens}. Let us now show that it is surjective.
Let $S = T_1\oplus\cdots\oplus T_r$ be a direct sum of~$r$ simple algebras, not necessarily central, all finite dimensional over $F$. Take a common Galois splitting field $E$; then $E\otimes (T_1\oplus\cdots\oplus T_r)$ is a direct sum of matrix algebras over~$E$, perhaps with more components (after scalar extension some~$T_i$ may decompose into more than one component if they are not central). Let $\M[n](E;c)$ be the skew matrix algebra of degree $n = \sum \deg T_i$, which is the radical envelope of the blocks defined by $E \tensor S$.

The skew set satisfies the conjugacy condition \eq{Gact} with respect to any $\tau \in \Gal(E/F)$, so we can extend the action of $\Gal(E/F)$ to the radical envelope as a permutation action on the matrix units. The invariant subalgebra~$A$, which is semiassociative over $F$, has $\s(A) = T_1 \oplus \cdots \oplus T_r$.
\end{proof}

\begin{defn}
A semiassociative algebra is {\bf{homogeneous}} if it is semicentral, and the atoms are all Brauer equivalent to each other.
\end{defn}
(So a semiassociative algebra is homogeneous if and only if $\s(A)$ is semicentral and homogeneous). All skew matrix algebras are homogeneous by \Cref{sumdegjac}.

We further define submonoids of $\MM{semi}$ as the inverse images of the monoids in \eq{SS}, as in the following diagram.
\begin{equation}\label{twolines}
\xymatrix@C=10pt@R=14pt{
\MM{mat} \ar@{->}[d] \ar@{-}@/^3ex/[rrrrrr]|{\Br[sa](F)} & \sub & \MM{hom} \ar@{->}[d]&  \sub & \MM{semicentral} \ar@{->}[d] & \sub & \MM{semi} \ar@{->}[d]^{\s} \\
\SS{split} & \sub & \SS{hom} & \sub & \SS{semicentral} & \sub & \SS{sep}
}
\end{equation}

Namely:
\begin{itemize}
\item $\MM{semicentral}$ is the family of semicentral semiassociative algebras, which are the algebras with central atoms, a monoid by \Pref{supernormalx}; %
\item $\MM{hom}$ is composed of the homogeneous algebras;
\item And $\MM{mat}$ is the family of skew matrix algebras, identifies by \Cref{use52} as the semiassociative algebras whose semisimple quotient of the nucleus belongs to $\SS{split}$.
\end{itemize}

\begin{exer}\label{88211}
Let $\varphi \co S \ra S'$ be a homomorphism of semigroups. Let $P' \sub S'$ be a pure subsemigroup. Then $\varphi^{-1}(P')$ is a pure subsemigroup of $S$.
\end{exer}

\begin{cor}\label{purim}
The inclusions $\MM{hom} \sub \MM{semicentral} \sub \MM{semi}$ are pure.
\end{cor}
\begin{proof}
Apply Exercise~\ref{88211} to \Prefs{pure1}{pure2}.
\end{proof}

We conclude that semicentrality is a property of the Brauer class:
\begin{prop}\label{classofsuper}
Every semiassociative algebra similar to a semicentral algebra is semicentral.
\end{prop}
\begin{proof}
If $A' \sim A$ where $A$ is semicentral, then $A'$ divides, in $\MM{semi}$, a semicentral algebra $A \tensor \M[n](F;c)$, so must be semicentral by purity.
\end{proof}

In order to connect this setup with the Brauer group, we should bring up the following fact.
\begin{exer}[First isomorphism theorem for monoids]\label{Isom1}
If the image of a homomorphism $\varphi \co M \ra N$ of commutative monoids is a group, then $\Im(\varphi) \isom M/\Ker(\varphi)$, where as usual $\Ker(\varphi) = \set{x \in M \suchthat \varphi(x) = 1}$.

(Hint: if $\varphi(x) = \varphi(y)$, take $r \in M$ with $\varphi(r) = \varphi(x)^{-1}$, so that $rx, ry$ are in the kernel and thus $x \sim y$).
\end{exer}

Let $\a \co \SS{hom} \ra \Br(F)$ be the map associating to an associative semisimple homogeneous algebra the Brauer class of its components, which are equivalent by assumption. The kernel of $\a$ is clearly $\SS{split}$, so by \Exeref{Isom1}, $\SS{hom}/ \SS{split} \isom \Br(F)$.

Now consider the composition $(\a \circ \s) \co \MM{hom} \ra \Br(F)$, sending a homogeneous semiassociative algebra to the Brauer class of its atoms. %
This map is onto because associative central simple algebras are in~$\MM{hom}$; and the kernel is $\MM{mat}$ by \Cref{use52}. It follows that
\begin{equation}\label{xx}
\MM{hom} / \MM{mat} \isom \Br(F),
\end{equation}
with the embedding $\Br(F) \sub \Br[sa](F)$ presented in~\eq{BrBr} obtained from $\MM{hom} \sub \MM{semi}$.

\forget
By dividing the upper line of the diagram above by $\MM{mat}$, we define an intermediate monoid,
\begin{equation}\label{Brchain}
\Br(F) \sub \Br[sc](F) \sub \Br[sa](F)
\end{equation}

as the submonoids containing the classes of associative, semicentral and arbitrary semiassociative algebras over $F$, respectively. In the same direction as above, we have that:
\begin{rem}
Every element of a class $\nacl{A} \in \Br[sc](F)$ is semicentral. Indeed if $A \sim A'$ where $A$ is semicentral, then $A'$ divides, in $\MM{semi}$, a semicentral algebra $A \tensor \M[n](F;c)$.
\end{rem}
\forgotten

\forget
Consider the classes
$$\MM{mat} \sub \MM{hom} \sub \MM{semicentral} \sub \MM{nonsingular} \sub \MM{semi}$$
of nonsingular skew matrix algebras,
semicentral semiassociative algebras whose simple component of the nucleus are all Brauer equivalent to each other,
all semicentral semiassociative algebras,
nonsingular semiassociative algebras, and all semiassociative algebras, respectively.
By
\Rref{tensormat}, %
\Cref{supernormalx}, %
and \Tref{basicTx}, %
respectively,
each class is closed under tensor product, so we obtain a chain of monoids. The three submonoids of $\MM{semi}$ are characterized in terms of the nucleus, by: having semisimple nucleus and thus having atoms; having central atoms; having only matrix algebras as atoms.
\forgotten

\begin{rem}
One would want to intersect the chain
\begin{equation}\label{MMchain}
\MM{mat} \sub \MM{hom} \sub \MM{semicentral} \sub \MM{semi}
\end{equation} defined above with the basic conditions lurking in the background. If we restrict to associative algebras there is still significant structure, since the map $\s \co \MM{semi} \ra \SS{sep}$ remains surjective, as seen in the proof of \Pref{Guypromisedaproof}. However, if we restrict attention to {\emph{simple}} associative algebras, the chain collapses: an associative simple semiassociative algebra is homogeneous, being its own single atom; and associative simple skew matrices are nothing but matrices; so in this class our Brauer monoid collapses to the Brauer group.
\end{rem}

\begin{rem}
In \Pref{goodcompare} we saw that our Brauer equivalence reduces to the classical equivalence for simple associative algebras. However, the zero matrix algebra of \Eref{assocnonsimple} is an associative skew matrix algebra which is not Morita equivalent to~$F$.
\end{rem}

\section{Underlying division algebras}\label{sec:underly}

We now prove that homogeneous semiassociative algebras have underlying associative division algebras.
\smallskip

A fundamental characterization of associative finite dimensional central simple algebras is that an algebra~$B$ is in this class if and only if whenever~$B$ is contained in a ring~$R$, the ring~$R$ decomposes as a tensor product of~$B$ and its centralizer \cite[Theorem~V.11.2]{Jac-sor}. The proof (in one direction) is to observe that $\End(B) = \op{B} \tensor B \sub \op{B} \tensor R$, and matrix algebras can be factored out. Matrix units in the nucleus provide the same factorization, so we have the following more general statement:
\begin{prop}\label{double}
Let $D$ be a central simple algebra contained in the nucleus of any nonassociative algebra $A$. Then $A \isom D \tensor \Ce[A]{D}$.
\end{prop}

This is nicely exemplified in \cite[Example~11]{BrPu}, where $A$ is a generalized nonassociative cyclic algebra and $D$ is an associative cyclic algebra.

\begin{prop}\label{underly}
Let $A$ be a homogeneous algebra, and $D$ the (associative) underlying division algebra of its atoms. Then there is a decomposition $A \isom D \tensor A'$ where $A'$ is a skew matrix algebra.
\end{prop}
\begin{proof}
Let $n = \deg(A)$ and $m = \deg(D)$. By assumption there is a split semisimple algebra $S \in \SS{split}$ such that $\s(A) = D \tensor S$. By \Lref{sumup}, applied to $A$, the sum of the degrees of the components of $S$ is equal to $n/m$. By \Pref{double} we have a decomposition $A = D \tensor \Ce[A]{D}$, so $F^{n/m} \sub S \sub \Ce[A]{D}$. Taking a maximal subfield $K \sub D$, $(K \tensor F^{n/m})^e$ acts faithfully on $A$ by \Pref{coverall}, which implies that $(F^{n/m})^e$ acts faithfully on $\Ce[A]{D}$. Since $\dim \Ce[A]{D} = (n/m)^2$, it follows that~$\Ce[A]{D}$ is a skew matrix algebra, as required.
\end{proof}

Notice that $D$ is uniquely determined by $A$, and may rightfully be referred to as the {\bf{underlying division algebra}} of $A$. In the language of monoids, we proved here that  $\MM{hom} = \MM{assoc} \cdot \MM{mat}$, where $\MM{assoc}$ is the monoid of associative central simple algebras over~$F$.

\begin{cor}
An associative division algebra $D$ is the unique member of minimal degree in the class $\nacl{D} \in \Br[sa](F)$ (so every other element has the form $D \tensor \M[n](F;c)$).
\end{cor}

In particular, similarly to the case in the associative Brauer group,
\begin{cor}
The trivial element $\nacl{F} \in \Br[sa](F)$ is the class of skew matrices.
\end{cor}
(This is not a-priori obvious, because $\Br[sa](F)$ contains every semiassociative algebra $A$ for which $A \tensor \M[n](F;c)$ is skew matrices for some $\M[n](F;c)$.)
\forget
\begin{prop}
The trivial element $\nacl{F} \in \Br[sa](F)$ is the class of nonsingular skew matrices.
\end{prop}
\begin{proof}
A semiassociative algebra $A$ is in $\nacl{F}$ if there are nonsingular skew sets $c,c'$ such that $A \tensor \M[n](F;c) \isom \M[n'](F;c')$, and a priori may not be a skew matrix algebra by itself. But as in \Tref{subgp}, since $\M[n'](F;c')$ is semicentral of weight~$1$, this is true for $A$ as well; but then the components of $\nuc{A}$ are all Brauer equivalent to $F$, in the classical sense, so $A$ splits by \Pref{use52}.
\end{proof}
\forgotten

\section{Invertibility in the Brauer monoid}\label{sec:final}

\forget
\begin{prop}\label{prime1}
$\MM{semicentral}$ is a prime submonoid of $\MM{semicentral}$.
\end{prop}
\begin{proof}
Let $S$ and $S'$ be atoms of $A$ and~$B$, respectively. The resulting atoms of $A \tensor B$ are the direct summands of $S \tensor S'$, but the center of each direct summand contains both $\Zent(S)$ and $\Zent(S')$, so if $A \tensor B$ has central atoms, so must $A$ and~$B$.
\end{proof}

For a finer analysis, let us define the {\bf{weight}} $\w(A)$ of a semicentral semiassociative algebra as its number of atoms, counted as elements in the Brauer group $\Br(F)$. For example, a nonsingular skew matrix algebra has weight~$1$, because its atoms are all in $[F]$.

\begin{prop}\label{onw}
Let $A,B$ be semicentral semiassociative algebras. Then $\w(A \tensor B) \geq \w(A)$.
\end{prop}
\begin{proof}
If $\set{\alpha_1,\dots,\alpha_t}$ are the Brauer classes of atoms of $A$, and $\beta$ is any Brauer class of an atoms of~$B$,
then, as in \Cref{supernormalx}, $\set{\alpha_1 \beta, \dots, \alpha_t \beta}$ are Brauer classes of atoms of $A \tensor B$.
\end{proof}

Let $\MM{hom}$ be the class of semicentral semiassociative algebras with $\w(A) = 1$. This is a submonoid, and we have the refinement
\begin{equation}\label{MM}
\MM{mat} \sub \MM{hom} \sub \MM{semicentral} \sub \MM{nonsingular} \sub \MM{semi}.
\end{equation}
{}It follows from \Pref{onw} that:
\begin{cor}\label{prime2}
The submonoid $\MM{hom}$ is prime in $\MM{semicentral}$.
\end{cor}
\forgotten

We now use purity to identify the invertible elements of $\Br[sa](F)$.

\begin{thm}\label{subgp}
The invertible elements in $\Br[sa](F)$ are precisely the elements of $\Br(F)$, embedded in $\Br[sa](F)$ by \eq{BrBr}.
\end{thm}
\begin{proof}
Let $A \in \MM{semi}$, such that $\nacl{A} \in \Br[sa](F)$ is invertible, with inverse $\nacl{B} \in\Br[sa](F)$.
As $\nacl{A} \nacl{B}=\nacl{A\tensor B} = \nacl{F}$, we may write
$$A\tensor  B\tensor \M[t](F;c)\cong\M[r](F;c')$$
for some $t,r\ge 1$ and skew sets $(c), (c')$ of the respective dimension. Since $\M[r](F;c')$ is semicentral, \Pref{purim} shows that $A$ and~$B$ are semicentral. Since, in fact, $\M[r](F;c') \in \MM{hom}$, it follows from \Pref{purim} that $A \in \MM{hom}$ as well. By \Pref{underly}, $A$ is equivalent to an associative division algebra, and so $\nacl{A} \in \Br(F)$.

\end{proof}
In particular,
\begin{cor}
$\Br(F)$ is the unique maximal subgroup of $\Br[sa](F)$.
\end{cor}

\section{A combinatorial invariant of $\Br[sa](F)$}\label{sec:concrete}

The map $\s \co \MM{semi} \ra \SS{sep}$ discussed in \Pref{Guypromisedaproof} takes the elements of $\MM{mat}$ to $\SS{split}$, and thus induces a map
$$\bar{\s} \co \Br[sa](F) \ra \SS{sep}/\SS{split},$$
which is onto by \Pref{Guypromisedaproof}. How do we work with the quotient at the right-hand side? The idea is to view the atoms of a separable algebra $S \in \SS{sep}$ up to Brauer equivalence, and then pass to distributions to accommodate for the tensor product with arbitrary skew matrices,  whose images are arbitrary sums of central matrices.

\smallskip

Let $\allBr$ denote the disjoint union of the Brauer groups $\Br(F')$, ranging over the finite separable extensions $F'/F$. Namely, simple finite dimensional $F$-algebras with separable centers, up to Brauer equivalence over the respective center. This is not a monoid with respect to the tensor product over~$F$, as the tensor product of simple non-central algebras is often not simple. Therefore, reminiscent of the Grothendieck ring of representations, we consider the commutative $\Q$-algebra $\Q[\allBr]$ which by definition is generated by $\allBr$, with multiplication
of
$\alpha' \in \Br(F')$ and $\alpha'' \in \Br(F'')$ defined by $\alpha' \alpha'' = \gamma_1+ \cdots + \gamma_t$, where the tensor product of representatives $A' \in \alpha'$ and $A'' \in \alpha''$ decomposes into simple summands as $A' \tensor[F] A'' = C_1 \oplus \cdots \oplus C_t$ and $\gamma_i$ is the Brauer class of $C_i$ over its center. In particular the group algebra $\Q[\Br(F)]$ is a (monomial) subalgebra. The elements $\frac{1}{[F':F]}[F']$, for the Galois extensions $F'/F$, are idempotents in~$\Q[\allBr]$.

There is a {\bf{central dimension}} function $\dim\!\co \Q[\allBr] \ra \Q$, assigning to each class $\alpha \in \Br(F')$ the dimension of the center, $\dimcol{F'}{F}$, and extended additively.
Since the dimension of the center is additive and multiplicative, the central dimension is an algebra homomorphism.

\begin{defn}
By $\Delta(F)$ we denote the space of finitely supported rational ``distributions'' over $\allBr$, namely, elements
$\sum q_i \alpha_i \in \Q[\allBr]$ of central dimension~$1$, with positive rational coefficients~$q_i$.
\end{defn}
This is a monoid under the multiplication in the $\Q$-algebra. Geometrically, $\Delta(F)$ is the (rational) convex hull of the union $\allBr$. Its elements can be viewed as distributions in the usual sense, if we accept that the ``probability'' to choose $\alpha_j$ in the formal sum $\sum q_j \alpha_j$ is $\dim(\alpha_j)q_j$. The product of elements is then, indeed, the convolution of distributions, properly interpreted.

\begin{exmpl}
Take $F = \R$. Then $\allBr[\R] = \Br(\R) \cup \Br(\C) = \set{[\R],[\HQ],[\C]}$. In the algebra $\Q[\allBr[\R]]$, $[\R]$ is the identity element, $[\HQ]^2 = [\R]$, $[\C][\HQ] = [\C]$ and $[\C]^2 = 2[\C]$ since $\C \tensor[\R] \C = \C \oplus \C$. The algebra is isomorphic to $\Q \oplus \Q \oplus \Q$ by sending $[\R] \mapsto (1,1,1)$, $[\HQ] \mapsto (1,-1,1)$ and $[\C] \mapsto (2,0,0)$. The central dimension is then the projection onto the first entry.

Restricting to $\Delta(\R)$, the map $q_0[\R]+q_1[\HQ]+(1-q_0-q_1)\frac{1}{2}[\C] \mapsto (1,q_0-q_1,q_0+q_1)$ defines an isomorphism of monoids from $\Delta(\R)$ to $\set{(1,x,y) \in \Q^3 \suchthat \abs{x}\leq y \leq 1}$, with entry-wise multiplication.

\end{exmpl}

Now define an epimorphism $$\pi \co \SS{sep} \ra \Delta(F)$$
by sending a direct sum of separable simple algebras $T_1 \oplus \cdots \oplus T_t$ to the distribution $\frac{1}{t} \sum \frac{[T_i]}{\dim(\Zent(T_i))}$, where $t$ is the number of atoms. The map is onto by clearing denominators and repeating each summand as necessary.
For example, $\pi(r[\R]+2c[\C]+h[\HQ]) = \frac{r}{n}[\R]+\frac{c}{n}[\C]+\frac{h}{n}[\HQ]$, where $n = r+2c+h$.

Clearly $\pi(S) = [F]$ for any $S \in \SS{split}$, which is why we bother with distributions in the first place. So the map is well-defined as an
 epimorphism of monoids
$$\pi \co \SS{sep}/\SS{split} \ra \Delta(F).$$

\begin{cor}
The map $\pi \circ \bar{\s} \co \Br[sa](F) \ra \Delta(F)$ is an epimorphism.
\end{cor}

To put this in context, let us extend the diagram \eq{twolines} one line further:
$$
\xymatrix@C=10pt@R=14pt{
\MM{mat} \ar@{->}[d] & \sub & \MM{hom} \ar@{->}[d]&  \sub & \MM{semicentral} \ar@{->}[d] & \sub & \MM{semi} \ar@{->}[d]^{\s} \\
\SS{split}\ar@{->}[d] & \sub & \SS{hom} \ar@{->}[d] & \sub & \SS{semicentral} \ar@{->}[d] & \sub & \SS{sep} \ar@{->}[d]^{\pi} \\
1 & \sub & \Br(F) & \sub & \Q[\Br(F)]^{(1)} & \sub & \Delta(F) \\
}$$
where $\Q[\Br(F)]^{(1)}$ is the monoid of distribution elements in the group algebra of $\Br(F)$.
In particular, these maps realize the isomorphisms $$\MM{hom}/\MM{mat} \isom \SS{hom}/\SS{split} \isom \Br(F).$$

\begin{cor}
Let $F$ be a field with nontrivial Brauer group. Then $\Br[sa](F)$ has elements of infinite order.
\end{cor}
(By ``infinite order'' we mean that the powers of the element are distinct, so the statement does not immediately follow from \Tref{subgp}.)
\begin{proof}
Let $C$ be a central simple algebra of index~$p$ over $F$, so that $[C]^p = [F]$ in $\Br(F)$. Let~$A$ be a semiassociative algebra with $\s(A) = F \oplus C$. The image of $A$ is $\Delta(F)$ is $\pi (\bar{\sigma})([A]) = \frac{1}{2}([F]+[C])$, which has infinite order, as can be seen by mapping the subalgebra of $\Q[\Br(F)]$ generated by $[C]$ to $\C$ by sending $[C] \mapsto \exp(\frac{2\pi i}{p})$.
\end{proof}

\forget
\begin{cor}
Let $F$ be a field with nontrivial $\Brp[2](F)$. Then $\Br[sa](F)$ has elements of infinite order.
\end{cor}
\begin{proof}
Let~$Q$ be a nonsplit associative quaternion algebra over $F$. Then $e = \frac{1}{2}([F]+[Q])$ is an idempotent in the group algebra $\Q[\Br(F)]$; let $e' = \frac{1}{2}([F]-[Q])$ be its complement; the central dimensions are~$1$ and~$0$, respectively. Let~$A$ be a semiassociative algebra with $\s(A) = F \oplus F \oplus Q$. The image of $A$ in $\Delta(F)$ is $\frac{1}{3}(2[F]+[Q]) = \frac{1}{3}(2(e+e')+(e-e')) = e+\frac{1}{3}e'$, which has infinite multiplicative order.
\end{proof}
\forgotten

\bibliographystyle{plain}
\bibliography{refs}

\end{document}